\documentclass[12pt]{amsart}
\usepackage[dvipdfmx]{graphicx}
\usepackage{color}
\usepackage{verbatim}
\usepackage{epic}
\usepackage{amsmath}
\usepackage{amssymb, mathrsfs}
\usepackage{amsbsy}
\usepackage{amscd}
\usepackage{amsthm}\usepackage{bm}

\newcommand{\N}{\mathbb{N}}
\newcommand{\R}{\mathbb{R}}
\newcommand{\Z}{\mathbb{Z}}
\newcommand{\E}{\mathbb{E}}
\newcommand{\set}[1]{\left\{#1\right\}}

\newcommand{\supp}{\mathop{\mathrm{supp}}\nolimits}
\renewcommand{\P}{\mathbb{P}}

\newcommand{\A}{\mathcal{A}}

\newtheorem{theorem}{Theorem}
\newtheorem*{theorem*}{Theorem}
\newtheorem{lemma}{Lemma}
\newtheorem{proposition}{Proposition}
\newtheorem*{proposition*}{Proposition}
\newtheorem{corollary}{Corollary}

\theoremstyle{definition}
\newtheorem{remark}{Remark}

\numberwithin{equation}{section}

\begin{document}

\title[Quenched tail estimate for RWRS]
{
Quenched tail estimate for the random walk in random scenery
and in random layered conductance}
\author{Jean-Dominique Deuschel}
\address[Jean-Dominique Deuschel]
{Institut f\"ur Mathematik, Technische Universit\"at Berlin, Berlin, Germany
and Research Institute in Mathematical Sciences, 
Kyoto University, Kyoto, Japan}
\email{deuschel@math.tu-berlin.de}
\author{Ryoki Fukushima}
\address[Ryoki Fukushima]
{Research Institute in Mathematical Sciences, 
Kyoto University, Kyoto, Japan}
\email{ryoki@kurims.kyoto-u.ac.jp}
\date{\today}

\begin{abstract}
We discuss the quenched tail estimates for the random walk in random scenery. 
The random walk is the symmetric nearest neighbor walk and the 
random scenery is assumed to be independent and identically distributed, 
non-negative, and has a power law tail. We identify the long time aymptotics 
of the upper deviation probability of the random walk in quenched random 
scenery, depending on the tail of scenery distribution  and the amount of 
the deviation. 
The result is in turn applied to the tail estimates for a random walk in 
random conductance which has a layered structure. 
\end{abstract}

\maketitle

\section{Introduction}
In this article, we study the random walk in random scenery, mainly in the
continuous time setting. 
This model is first introduced and studied in the discrete time setting by 
Borodin~\cite{Bor79a,Bor79b} and Kesten--Spitzer~\cite{KS79} independently. 
It is the sum of independent and identically distributed random variables 
$(\{z(x)\}_{x\in\Z^d},\P)$ along a random walk $((S_n)_{n\in\Z_+},P_0)$ 
starting at the origin:
\begin{equation}
W_n=\sum_{k=1}^n z(S_k).
\end{equation}
{One of the motivation in~\cite{KS79}} was to construct a
new class of self-similar processes as a scaling limit of this process 
under the joint law. {They proved that, when $d=1$ and $z$ and $S$ are 
centered and belong
to the domain of attraction of stable law with index $\alpha\in(0,2]$ and
$\beta\in (1,2]$ respectively, the rescaled process
\begin{equation}
 \left(n^{-(1-\frac{1}{\beta}+\frac{1}{\alpha\beta})}W_{\lceil nt \rceil}
\right)_{t\ge 0}\label{KS-scaling}
\end{equation}
converges in distribution under $\P\otimes P_0$.}
Subsequently, a lot of works
have been done to extend the distributional limit theorems, 
not only under the joint law but also for the quenched scenery, namely for
almost all realizations of the scenery, and to obtain law of iterated 
logarithm type results. 
We refer the reader to the introduction of~\cite{GPPdS14} for more 
details and background. 

In this paper, we discuss upper tail estimates with quenched scenery 
which is assumed to be non-negative and have a power law tail.
The tail estimate for the random walk in random scenery is studied rather
recently. As for the annealed tail estimates, that is, under the joint law, 
there are extensive results. 
In the case of Gaussian scenery, for the Brownian 
motion~\cite{CP01} and for the Markov chains satisfying the level-2 full 
large deviation principle~\cite{GP02} instead of the random walk, the
 large deviation principles for $t^{-\frac32}W_t$ and $n^{-\frac32}W_n$ are 
proved. Later 
the moderate deviations are also studied for the Brownian 
motion in $d\le 3$~\cite{Cas04}. The paper~\cite{KL98} also includes an 
upper tail estimate for the stable process in Brownian scenery, which is a 
continuous space counterpart. 
The bounded scenery case is studied in~\cite{Rem00,AC03b} for the
Brownian motion and the large deviation 
principle for $t^{-1}W_t$ is established. 
More recently, the random walk in a random scenery with stretched or 
compressed exponential tail attracted much attention, partially in 
relation to the recent development on the tail estimates for the 
self-intersection local time~\cite{Che10,Kon10}. 
There are various results depending on the tail and the spatial 
dimensions~\cite{AC07b,AC07a,GKS07,GHK06}.
Later, under the Cram\'er condition for the upper tail 
and the finiteness of third moment,~\cite{FMW08} proved quite precise moderate
deviation estimates. Also in one dimensional case, the moderate deviations
are studied for stable random walk in a scenery with sub-Gaussian 
tail~\cite{Li12}.

Concerning the quenched tail estimates, there are not many results. In the
case of one dimensional Brownian motion in Gaussian scenery~\cite{AC03a} 
and bounded scenery~\cite{AC03b}, the large deviation principles are 
proved for $t^{-1}(\log t)^{-\frac12}W_t$ and $t^{-1}W_t$ respectively. 
And again~\cite{Cas04} extends the result to the moderate deviations 
for the Brownian motion in Gaussian scenery for $d\le 3$. For more on the 
technical correspondence, see Remark~\ref{rem:AC} below.
One may also consider the tail behavior under the scenery law with the 
random walk trajectory fixed. 
In this case the large deviation principle is proved in~\cite{GP02}. 

The random walk in random scenery is naturally interpreted as a 
diffusing particle in a layered random media. See for example the 
introduction of~\cite{AC03a} where it is studied under the name 
``diffusion in a random Gaussian shear flow drift''. 
In this paper, we study another model --- the random walk in a random
layered conductance introduced and studied in~\cite{ADS16}. 
Here again we assume that the random conductance has a power law tail.
See Section~\ref{sect:RCM} below for the precise setting.
This process can be described as the time change
of the simple random walk with the clock process given by the random
walk in random scenery. By using the tail estimates for the latter, 
we derive asymptotics of the super-diffusive deposition probabilities. 

\section{Setting and results}
Throughout the paper, $(\{z(x)\}_{x\in\Z^d},\P)$ is a family 
of non-negative, independent and  identically distributed random variables 
with the distribution satisfying
\begin{equation}
 \P(z(x)> r)= r^{-\alpha+o(1)} \textrm{ as }r\to\infty
\label{ass-tail}
\end{equation}
for some $\alpha>0$.
\begin{remark}
The non-negativity could be replaced by some assumption on the lower tail.
For instance, $z$ being bounded from below would suffice for the results in
Section~\ref{sect:RWRS}. We work with non-negative scenery with an application
to the random conductance model (Section~\ref{sect:RCM}) in mind.
\end{remark}
\subsection{Random walk in random scenery}
\label{sect:RWRS}
Let $(\{S_t\}_{t\ge 0},P_x)$ be the continuous time simple random walk
on $\Z^d$ starting at $x\in\Z^d$. 
We consider the additive functional defined by
\begin{equation}
A_t= \int_0^t z(S_u){\rm d} u.
\end{equation}
This is a natural continuous time analogue of the random walk in 
random scenery. 
\begin{theorem}
 \label{RWRS}
Let $\rho>0$. Then $\P$-almost surely, 
\begin{equation}
P_0\left(A_t \ge t^{\rho}\right)
= \exp\left\{-t^{p(\alpha,\rho)+o(1)}\right\}
\label{rwrs}
\end{equation}
as $t\to\infty$, where 
\begin{equation}
 p(\alpha,\rho)=
\begin{cases}
\frac{2\alpha\rho}{\alpha+1}-1,
 &\rho\in\left(\frac{\alpha+1}{2\alpha}\vee 1,
\frac{\alpha+1}{\alpha}\right],\\[5pt]
\alpha(\rho-1),&\rho>\frac{\alpha+1}{\alpha}
\end{cases}
\label{def-p}
\end{equation} 
for $d=1$ and 
\begin{equation}
 p(\alpha,\rho)=
\begin{cases}
\frac{2\alpha\rho-d}{2\alpha+d},
 &\rho\in\left(\frac{d}{2\alpha}\vee 1,
\frac{\alpha+d}{\alpha}\right],\\[5pt]
\frac{\alpha(\rho-1)}{d},&\rho>\frac{\alpha+d}{\alpha}
\end{cases}
\end{equation} 
for $d\ge 2$. 
Furthermore, when $d=1$ and $\rho< \frac{\alpha+1}{2\alpha}\vee 1$
or $d\ge 2$ and $\rho< \frac{d}{2\alpha}\vee 1$, 
$\P$-almost surely the above probability is bounded from 
below by a negative power of $t$. 
\end{theorem}
{Figure~\ref{p} shows the phase diagram of $p(\alpha,\rho)$. 
The graph of $p(\alpha,\rho)$ in the same case can 
be found in Figures~\ref{a<1} and~\ref{a>1} in subsection 5.1.}
\begin{figure}[h]
 \begin{picture}(300,220)(-55,-13)
 \put(-10,-10){$0$}
 \put(0,0){\vector(1,0){180}}
 \put(0,0){\vector(0,1){180}}
 \put(185,0){$\alpha$}
 \put(0,185){$\rho$}
 \put(45,-13){{$1\vee \frac{d}{2}$}}
 \dashline{2}(0,50)(50,50)
 \put(-13,45){$1$}
 \dashline{2}(50,0)(50,50)
 \qbezier(50,50)(20,60)(10,170)
 \qbezier(20,170)(60,70)(180,60)
 \put(50,50){\line(1,0){130}}
 \put(30,20){$0$}
 \put(40,75){{$(\frac{2\alpha\rho}{\alpha+1}-1)1_{d=1}$}}
 \put(50,57){{$+\frac{2\alpha\rho-d}{2\alpha+d}1_{d\ge 2}$}}
 \put(90,120){{$\frac{\alpha(\rho-1)}{d}$}}
 \end{picture}
\caption{The phase diagram of the exponent $p(\alpha,\rho)$. 
The curves are $\rho={(\frac{\alpha+1}{2\alpha}1_{d=1}+\frac{d}{2\alpha}1_{d\ge 2})}\vee 1$ and 
$\rho=\frac{\alpha+{d}}{\alpha}$.}
\label{p} 
\end{figure}
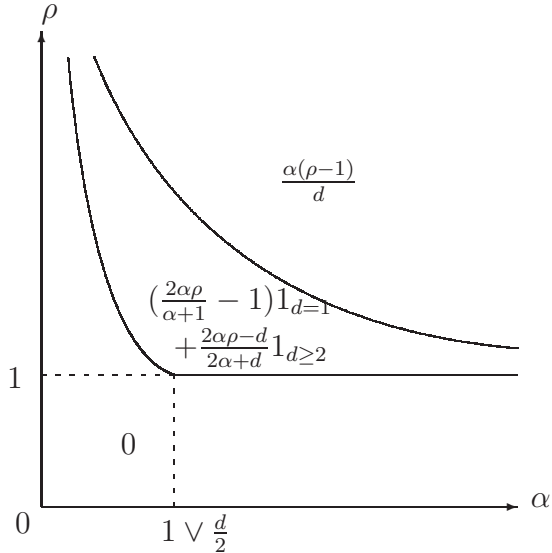
\begin{remark}
We can prove the corresponding results for the discrete time random
walk in random scenery by the same argument. The exponent in the first regime
is the same but in the second regime it would be $\infty$. This is because
the random walk needs to travel much further distance 
than $t$ to achieve $A_t\ge t^\rho$ in the second regime. 
See Section~\ref{spec} below.
\end{remark}
\begin{corollary}
Let $\alpha\le 1$. Then 
{$\P\otimes P_0$}-almost surely, 
\begin{equation}
 \limsup_{t\to\infty}\frac{\log A_t}{\log t}\le
\begin{cases}
 \frac{\alpha+1}{2\alpha},&d=1,\\[5pt]
 \frac{d}{2\alpha},&d\ge 2.
\end{cases}
\label{upper-rate}
\end{equation}
\end{corollary}
\begin{remark}
This follows immediately from Theorem~\ref{RWRS} and the Borel--Cantelli
Lemma and is, 
to the best of our knowledge, new. The exponent in the one dimensional
case coincides with the self-similar parameter identified in~\cite{KS79},
see~\eqref{KS-scaling}. When $d\ge 2$, the random walk visits a single 
site not too many times and hence $A_t$ should behave in a similar way 
to the sum of independent and identically distributed random variables, in 
which case the scaling exponent is $1/\alpha$. Based on this observation, 
we believe that \eqref{upper-rate} is not sharp for $d\ge 3$.
We focus on $\alpha\le 1$ since in the other case $\alpha>1$, 
one can prove much finer aymptotics 
\begin{equation}
\lim_{t\to\infty}\frac{1}{t}A_t=\E[z(0)]
 \label{LLN}
\end{equation}
{$\P\otimes P_0$}-almost surely, for example by using Kakutani's random 
ergodic theorem~\cite{Kak51} as pointed out in~\cite{dHS06}. 
In one-dimensional case with stronger 
moment conditions, even finer asymptotics, the law of iterated 
logarithm for $A_t-t\E[z(0)]$, is shown in~\cite{CKS99,Zha01}. 
\end{remark}
When $d=1$, $\alpha\le 1$ and $\rho=\frac{\alpha+1}{2\alpha}$, 
the bound~\eqref{rwrs} with $p(\alpha,\rho)=0$ holds by monotonicity. 
The same applies also for $d\ge 2$, $\alpha\le \frac{d}2$ and 
$\rho=\frac{d}{2\alpha}$.
On the other hand, when $d=1$, $\alpha>1$ and $\rho=1$ or 
$d\ge 2$, $\alpha>\frac{d}2$ and $\rho=1$, the tail of 
$P_0\left(A_t \ge ct\right)$ depends on $c$. 
This tends to one for $c<\E[z(0)]$ whereas for $c>\E[z(0)]$, 
this is the standard large deviation regime. We can get a lower bound 
by extrapolation and in fact it is the correct tail behavior. 
\begin{theorem}
 \label{LDP}
Let $d=1$ and $\alpha>1$ or $d\ge 2$ and $\alpha>\frac{d}{2}$. 
Then for any $c>\E[z(0)]$, $\P$-almost surely, 
\begin{equation}
P_0\left(A_t \ge ct\right)
= 
\begin{cases}
\exp\left\{-t^{\frac{\alpha-1}{\alpha+1}+o(1)}\right\},&d=1,\\[5pt]
\exp\left\{-t^{\frac{2\alpha-d}{2\alpha+d}+o(1)}\right\},&d\ge 2 
\end{cases}
\label{ldp}
\end{equation}
as $t\to\infty$.
\end{theorem}

\subsection{Random walk in layered random conductance}
\label{sect:RCM}
The model considered here is the random walk in random conductance in 
$\Z^{1+d}$ which is constant along lines parallel to an axis. 
We write $x\in\Z^{1+d}$ as $(x_1, x_2)$ with $x_1\in\Z$ and $x_2\in \Z^d$. 
In what follows, any points in $\R^k\setminus\Z^k$ ($k\in\N$) 
are to be understood as closest lattice points. 
Let $((X_t)_{t\ge 0},P^\omega_x)$ be a continuous time Markov chain with 
jump rates
\begin{equation}
  \omega(x,x\pm \mathbf{e}_i)= 
 \begin{cases}
  z(x_2),& i=1,\\
  1,& i\ge 2.
 \end{cases}
\end{equation}
Such a process is sometimes called the variable speed random walk in the 
random conductance field $\omega$. 
This is related to the random walk in random scenery as follows: 
Let $(S^1,S^2)$ be the continuous time simple random walk on $\Z^{1+d}$. 
Then we have a representation
\begin{equation}
 (X^1_t, X^2_t)_{{t\ge 0}}=(S^1_{A^2_t}, S^2_t)_{{t\ge 0}},
\end{equation}
where the clock process is defined by $A^2_t=\int_0^t z(S^2_u){\rm d}u$. 

This model is introduced in~\cite{ADS16}, Example~1.11 in order to 
demonstrate an anomalous tail behavior of the transition probability 
and an upper bound is obtained.
Define a random distance --- sometimes called the chemical distance ---
by
\begin{equation}
d^\omega(x,y)=\inf_{\gamma\in\Gamma(x,y)}\sum_{e\in \gamma}
\frac{1}{\sqrt{\omega(e)}\vee 1}, 
\end{equation}
where $\Gamma(x,y)$ denotes the set of nearest neighbor paths connecting
$x$ and $y$. Then 
the following upper bound is proved in~\cite{ADS16}, Theorem~1.10:
\begin{equation}
 P^\omega_x(X_t=y)\le ct^{-\frac{d}2}
\begin{cases}
 \exp\left\{-c^{-1}\frac{d^\omega(x,y)^2}{t}\right\},&t>d^\omega(x,y),\\[5pt]
 \exp\left\{-c^{-1}d^\omega(x,y)\left(1\vee \log \frac{d^\omega(x,y)}{t}
\right)\right\},&t\le d^\omega(x,y).
\end{cases}
\label{HK-general}
\end{equation}
We take $x=0$, $y=t^{\delta}\mathbf{e}_1+t^\gamma\mathbf{e}$ with 
$\gamma\ge 0$, $\delta\in (\frac12,\infty)$, and $\mathbf{e}$ being a unit
vector orthogonal to $\mathbf{e}_1$, that is, we are looking at a 
super-diffusive deposition. It is not difficult to show that
\begin{equation}
 d^\omega(0,t^\delta\mathbf{e}_1+t^\gamma\mathbf{e})=
t^{\frac{2\delta\alpha}{2\alpha+d}\vee\gamma+o(1)}, 
\textrm{ as }t\to\infty
\label{asymp-d}
\end{equation}
$\P$-almost surely (see Lemma~1.12 in~\cite{ADS16} for the case $d=1$).
Then \eqref{HK-general} yields a stretched-exponential upper bound
\begin{equation}
 P^\omega_0(X_t=t^\delta\mathbf{e}_1+t^\gamma\mathbf{e})
\le 
 \exp\left\{-t^{\left(\frac{4\delta\alpha}{2\alpha+d}-1\right)\vee(2\gamma-1)+o(1)}\right\}
\end{equation}
when $\frac{2\delta\alpha}{2\alpha+d}\vee\gamma\le 1$. 
We are interested in whether this captures the correct asymptotics. 
The following theorem shows that it is sharp if and only if 
$\delta=\frac{2\alpha+d}{2\alpha}$ and $\gamma\le 1$. 
Note that in this case, we have
$d^\omega(0,t^\delta\mathbf{e}_1+t^\gamma\mathbf{e})=t^{1+o(1)}$ 
which corresponds to the boundary case of~\eqref{HK-general}. 

\begin{theorem}
\label{exp}
Let $\gamma,\delta\ge 0$ and $\mathbf{e}$ be a unit vector orthogonal 
to $\mathbf{e}_1$. 
Then for $\P$-almost every $\omega$, 
\begin{equation}
  P^\omega_0(X_t=t^\delta\mathbf{e}_1+t^\gamma\mathbf{e})=
 \exp\left\{-t^{q(\alpha,\delta)\vee((2\gamma-1)\wedge\gamma)+o(1)}\right\}
\end{equation}
as $t\to \infty$, where with the convention $\frac{\alpha}0=\infty$, 
for $d=1$,
\begin{equation}
 q(\alpha,\delta)=
\begin{cases}
0,&\delta< \frac{1}{2}\vee\frac{\alpha+1}{4\alpha},\\[5pt]
2\delta-1,& \delta\in \left[\frac{1}{2},
\frac{\alpha}{\alpha+1}\right),\\[5pt]
\frac{4\alpha\delta-\alpha-1}{3\alpha+1},
&\delta\in [\frac{\alpha}{\alpha+1}\vee\frac{\alpha+1}{4\alpha},
\frac{2\alpha+1}{2\alpha}],\\[5pt]
\frac{\alpha(2\delta-1)}{\alpha+1},
&\delta\in (\frac{2\alpha+1}{2\alpha},\frac{\alpha}{(\alpha-1)_+}),\\[5pt]
\delta,&\delta\ge \frac{\alpha}{(\alpha-1)_+}.
\end{cases}\label{q}
\end{equation}
The second case is to be understood as void for $\alpha<1$.
For $d\ge 2$, 
\begin{equation}
 q(\alpha,\delta)=
\begin{cases}
0,&\delta< \frac{1}{2}\vee\frac{d}{4\alpha},\\[5pt]
2\delta-1,& \delta\in \left[\frac{1}{2},
\frac{2\alpha}{2\alpha+d}\right),\\[5pt]
\frac{4\alpha\delta-d}{4\alpha+d},
&\delta\in [\frac{2\alpha}{2\alpha+d}\vee\frac{d}{4\alpha},
\frac{2\alpha+d}{2\alpha}],\\[5pt]
\frac{\alpha(2\delta-1)}{\alpha+d},
&\delta\in (\frac{2\alpha+d}{2\alpha},\frac{\alpha}{(\alpha-d)_+}),\\[5pt]
\delta,&\delta\ge \frac{\alpha}{(\alpha-d)_+}.
\end{cases}
\label{q-hd}
\end{equation}
with the second case understood as void for $\alpha<\frac{d}2$.
\end{theorem}
\begin{theorem}
\label{alg}
{Let $\gamma\le \frac{1}{2}$ and $\mathbf{e}$ be a unit vector 
orthogonal to $\mathbf{e}_1$.} Suppose $d=1$ and 
$\delta<\frac{1}{2}\vee\frac{\alpha+1}{4\alpha}$ or
$d\ge 2$ and $\delta<\frac12\vee \frac{d}{4\alpha}$. 
Then, $\P$-almost surely,  
$P^\omega_0(X_t=t^\delta\mathbf{e}_1{+t^\gamma}\mathbf{e})$ is 
bounded from below by a negative power of $t$. 
\end{theorem}
Figures~\ref{q:a<1} and~\ref{q:a>1} show the phase diagrams of the 
displacement exponents 
$q(\alpha,\delta)\vee((2\gamma-1)\wedge\gamma)$ in the one-dimensional
case. 
\begin{figure}[h]
 \begin{picture}(300,180)(-55,-13)
 \put(-10,-10){$0$}
 \put(0,0){\vector(1,0){220}}
 \put(0,0){\vector(0,1){150}}
 \put(225,-10){$\delta$}
 \put(0,155){$\gamma$}
 \put(65,-13){$\frac{\alpha+1}{4\alpha}$}
 \put(70,0){\line(0,1){50}}
 \put(145,-13){$\frac{2\alpha+1}{2\alpha}$}
 \put(150,0){\line(0,1){100}}
 \put(-13,50){$\frac12$}
 \put(0,50){\line(1,0){70}}
 \put(-10,100){1}
 \put(0,100){\line(1,0){150}}
 \qbezier(70,50)(110,75)(150,100)
 \qbezier(150,100)(175,124)(200,148)
 \put(30,20){$0$}
 \put(90,30){$\frac{4\alpha\delta-\alpha-1}{3\alpha+1}$}
 \put(165,50){$\frac{\alpha(2\delta-1)}{\alpha+1}$}
 \put(15,70){$2\gamma-1$}
 \put(60,120){$\gamma$}
 \end{picture}
\caption{The phase diagram of the displacement exponents in the case $d=1$ 
and $\alpha\le 1$. The slopes of the increasing piecewise linear curve 
starting at $(\frac{\alpha+1}{4\alpha},\frac12)$ are 
$\frac{2\alpha}{3\alpha+1}$ and $\frac{2\alpha}{\alpha+1}$ from the left.}
\label{q:a<1} 
\end{figure}
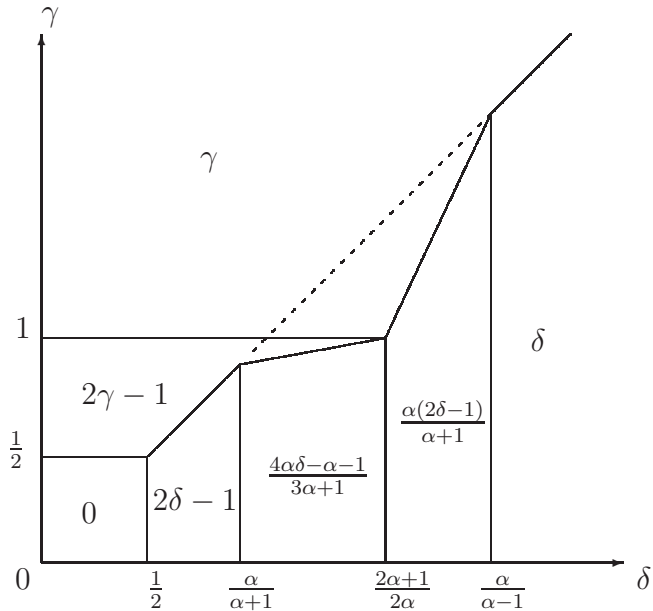
\begin{figure}[h]
 \begin{picture}(300,230)(-55,-13)
 \put(-10,-10){$0$}
 \put(0,0){\vector(1,0){220}}
 \put(0,0){\vector(0,1){200}}
 \put(225,-10){$\delta$}
 \put(0,205){$\gamma$}
 \put(75,0){\line(0,1){75}}
 \put(70,-13){$\frac{\alpha}{\alpha+1}$}
 \put(40,-13){$\frac12$}
 \put(40,0){\line(0,1){40}}
 \put(125,-13){$\frac{2\alpha+1}{2\alpha}$}
 \put(130,0){\line(0,1){85}}
 \put(165,-13){$\frac{\alpha}{\alpha-1}$}
 \put(170,0){\line(0,1){170}}
 \put(-13,40){$\frac12$}
 \put(0,40){\line(1,0){40}}
 \put(-10,85){1}
 \put(0,85){\line(1,0){130}}
 \qbezier(75,75)(58,58)(40,40)
 \qbezier(75,75)(103,80)(130,85)
 \qbezier(130,85)(150,128)(170,170)
 \qbezier(170,170)(185,185)(200,200)
 \dashline{2}(80,80)(170,170)
 \put(15,15){$0$}
 \put(42,20){$2\delta-1$}
 \put(85,30){$\frac{4\alpha\delta-\alpha-1}{3\alpha+1}$}
 \put(135,50){$\frac{\alpha(2\delta-1)}{\alpha+1}$}
 \put(185,80){$\delta$}
 \put(15,60){$2\gamma-1$}
 \put(60,150){$\gamma$}
 \end{picture}
\caption{The phase diagram of the displacement exponents in the case $d=1$ 
and $\alpha> 1$. The slopes of the increasing piecewise linear curve starting
at $(\frac12,\frac12)$ are 1, $\frac{2\alpha}{3\alpha+1}$, 
$\frac{2\alpha}{\alpha+1}$ and 1 from the left.}\label{q:a>1} 
\end{figure}

In the heat kernel estimate for random walks in random environment, 
it is sometimes useful to give a probabilistic control on the 
range of time (or space) in which an estimate like Theorem~\ref{exp} 
holds true. It turns out that we cannot hope much in this model.
For $r>0$, define
\begin{equation}
 \tau_r=\sup\left\{t\ge 0\colon 
 \left|\frac{1}{\log t}\log\left|
\log P^\omega_0(X_t=t^\delta\mathbf{e}_1+t^\gamma\mathbf{e})
\right|-
 q(\alpha,\delta)\right|>r\right\}.
\end{equation}
that is, the last time when the upper or lower bound is violated by $r$
in the exponent. 
\begin{theorem}
\label{range} 
Suppose $q(\alpha,\delta)>(2\gamma-1)\wedge\gamma$. 
 Then in the third or fourth regimes in~\eqref{q} and~\eqref{q-hd}, 
 for sufficiently small $r>0$, 
 \begin{equation}
 \P(\tau_r\ge t){\ge}t^{-r(C_1+o(1))}
 \label{tau}
 \end{equation}
 as $t\to\infty$, where 
 \begin{equation}
 C_1=
 \begin{cases}
 \frac{\alpha+1}{2},  &d=1 \textrm{ and }\delta\in(\frac{\alpha}{\alpha+1}\vee\frac{\alpha+1}{4\alpha},\frac{2\alpha+1}{2\alpha}]\\[5pt]
 \alpha+\frac{d}{2}, &d\ge 2 \textrm{ and }\delta\in(\frac{2\alpha}{2\alpha+d}\vee\frac{d}{4\alpha},\frac{2\alpha+d}{2\alpha}],\\[5pt]
 d, &d\ge 1\textrm{ and }
 \delta\in (\frac{2\alpha+d}{2\alpha},\frac{\alpha}{(\alpha-d)_+}).
 \end{cases}
 \end{equation}
 In the second regime in~\eqref{q} and~\eqref{q-hd}, 
 for sufficiently small $r>0$, 
 \begin{equation}
 \P(\tau_r\ge t)
 {\ge}
 \begin{cases}
 t^{-\alpha-\delta(\alpha-1)-\frac{r}{2}(3\alpha+1))+o(1)},&d=1,\\[5pt]
 t^{-2\alpha+\delta(2\alpha+d)-\frac{r}{2}(4\alpha+d))+o(1)},&d\ge 2.
 \end{cases}
 \label{tau3}
 \end{equation}
\end{theorem}
{\begin{remark}
We only stated the lower bounds since the point of 
Theorem~\ref{range} is the \emph{slow} decay of $\P(\tau_r\ge t)$. 
With some additional effort, it is possible to find the matching upper bound
by largely repeating the argument for Theorem~\ref{RWRS}. 
\end{remark}}
\begin{remark}
We exclude the case $q(\alpha,\delta)\le (2\gamma-1)\wedge\gamma$ since
there the heat kernel bound has little to do with $z$-field. 
In particular, the only possibility is that the lower bound is violated
and it must be because of atypically small values of $z$. Based on this 
observation, one can show that for sufficiently small $r$, 
$\P(\tau_r\ge t)$ is bounded by $\exp\{-t^{C_2}\}$ for some $C_2>0$. 
\end{remark}
{\section{A bound on the continuous time random walk}
We frequently use the following estimate on the transition probability of the
continuous time simple random walk
$p_t(x,y)=P_x(S_t=y)$.
This can be found in~\cite{DD05},
Proposition~4.2 and~4.3. 
\begin{lemma}
\label{RWHK}
There exist positive constants $c_1$--$c_4$ such that when $t\ge 1$, 
\begin{equation}
c_1t^{-\frac{d}2}\exp\left\{-c_2\frac{|x|^2}t\right\}
\le p_t(0,x)\le
c_3t^{-\frac{d}2}\exp\left\{-c_4\frac{|x|^2}t\right\}
\end{equation}
for $|x|\le t$ and 
\begin{equation}
 \exp\left\{-c_2|x|\left(1\vee\log\frac{|x|}t\right)\right\}
\le p_t(0,x)\le
 \exp\left\{-c_4|x|\left(1\vee\log\frac{|x|}t\right)\right\}
\end{equation}
for $|x|> t$.
\end{lemma}
}
\section{Proofs for the random walk in random scenery}
Both for the lower and upper bounds, we shall restrict our attention 
to the event 
\begin{equation}
 \set{\max_{0\le u\le t}|S_u|\le t^\mu}
\label{restriction}
\end{equation}
for a certain $\mu\ge \frac12$. Note that $\P$-almost surely,
\begin{equation}
 \max_{|x|\le t^\mu}z(x)= t^{\frac{d\mu}{\alpha}+o(1)}
\label{extreme}
\end{equation}
and the point of maximum $x_{\max{}}$ of $z$ within $|x|\le t^\mu$
satisfies $|x_{\max{}}|= t^{\mu+o(1)}$ as $t\to\infty$.
\subsection{Lower bounds of Theorem~\ref{RWRS}}
{We prove the lower bounds for the probability with an additional 
terminal constraint 
\begin{equation}
 P_0(A_t\ge t^\rho, S_t=0)\ge \exp\left\{-t^{p(\alpha,\rho)+o(1)}\right\},
\label{bridge}
\end{equation}
with the right hand side replaced by a negative power of $t$ when 
{$p(\alpha,\rho)=0$, i.e.,}~$\rho<\frac{\alpha+1}{2\alpha}\vee 1$ ($d=1$) or
$\rho<\frac{d}{2\alpha}\vee 1$ ($d\ge 2$). 
We will use these bounds in the proof of {Theorems}~\ref{exp} and~\ref{alg}.}
The basic strategy for the lower bound is simple. We let the random
walk go to $x_{\max{}}$ within time $\frac{t}{{4}}$, 
leave the right amount of local time there {by the time $\frac{t}2$ and 
then come back to the origin}. 
In the second regime where
$\rho>\frac{\alpha+d}{\alpha}$, the random walk stays
at $x_{\max{}}$ all the time {in $[\frac{t}4,\frac{t}2]$} and its probability is 
easy to evaluate.
But in the other regime, the optimal strategy is to leave the local time
much smaller than $\frac{t}{{4}}$ and we need a moderate deviation estimate
for the local time in~\cite{Che01}. Throughout the proof of the lower
bound, we assume~\eqref{extreme}.\\

\noindent\underline{\textsc{The second regime
$\rho>\frac{\alpha+d}{\alpha}$}}: 
In this case, we choose $\mu=\frac{\alpha(\rho-1)}{d}+\epsilon> 1$ 
so that $\rho<\frac{d\mu}{\alpha}+1$.
Then the above strategy of the random walk yields 
$A_t\ge\frac{1}{{4}}t^{\frac{d\mu}{\alpha}+1}>t^\rho$. 
By using {the Markov property and Lemma~\ref{RWHK}, we have 
\begin{equation}
\begin{split}
 & P_0\left(A_t\ge t^\rho, S_t=0\right)\\
 &\quad\ge P_0\left(S_{\frac{t}{{4}}}=x_{\max{}}\right)
 P_{x_{\max{}}}\left(S_u=x_{\max{}}\textrm{ for any }u\in
{\left[0,\frac{t}4\right]} \right)\\
 &\quad\qquad  \times P_{x_{\max{}}}\left(S_\frac{t}{2}=0\right)\\
 &\quad\ge \exp\left\{-t^{\mu+\epsilon}-\frac{t}4\right\}
\end{split}
\end{equation}}
{for sufficiently large $t$,} which is the desired bound. \\

\noindent\underline{\textsc{The first regime
$\rho\in \left(\left(\frac{\alpha+1}{2\alpha}1_{d=1}+
\frac{d}{2\alpha}1_{d\ge 2}\right)\vee 1,\frac{\alpha+d}{\alpha}\right]$}}: 
In this case, the right hand side of \eqref{rwrs} is sub-exponential 
and hence we must choose $\mu\le 1$. {We introduce a stopping time
\begin{equation}
 \zeta_t=\inf\set{u\ge 0\colon \ell_u(x_{\max{}})
\ge \frac{t^\rho}{z(x_{\max{}})}}
\end{equation}
where $\ell$ denotes the occupation time of the 
random walk. Note that $S_{\zeta_t}=x_{\max{}}$ almost surely.
Then, the strong Markov property and Lemma~\ref{RWHK} show
\begin{equation}
\begin{split}
& P_0\left(A_t\ge t^\rho,S_t=0\right)\\
&\quad \ge P_0\left(S_{\frac{t}{{4}}}=x_{\max{}}\right)
 E_{x_{\max{}}}\left[p_{\frac{3t}4-\zeta_t}(x_{\max{}},0) 
1_{\{\zeta_t\le {\frac{t}{4}}\}}\right]\\
&\quad \ge \exp\left\{-t^{2\mu-1+\epsilon}\right\}
P_0\left(\ell_{\frac{t}{{4}}}(0)\ge t^{\rho-\frac{d\mu}{\alpha}+\epsilon}\right)
\end{split}
\end{equation}}
{for sufficiently large $t$.}
In order to bound the second factor, we employ the following result:
\begin{theorem}[Theorem~1 in~\cite{Che01}]
\label{Chen}
Suppose $f$ is a nonnegative bounded function with a finite support 
containing the origin and
\begin{equation}
{\lim_{t\to\infty}\sup_{x\in\supp f}\left|\frac1{a(t)} E_x\left[\int_0^t f(S_u){\rm d}u\right]-1\right|=0.}
\end{equation}
Let ${1}\ll b(t)\ll t$. Then for $\lambda>0$, 
\begin{equation}
\limsup_{t\to\infty} \frac{1}{b(t)}\log
P_0\left(\int_0^tf(S_u){\rm d}u\ge \lambda a\left(\frac{t}{b(t)}\right)b(t)\right)
\le -1-\log \frac{\lambda}{4}
\end{equation}
and for $\lambda\in(0,1)$,
\begin{equation}
\liminf_{t\to\infty} \frac{1}{b(t)}\log
P_0\left(\int_0^tf(S_u){\rm d}u\ge \lambda a\left(\frac{t}{b(t)}\right)b(t)\right)
\ge -\log \frac{1+\lambda}{1-\lambda}.
\end{equation}
\end{theorem}
\begin{remark}
\label{Chen-ver}
An inspection of the proof in~\cite{Che01} shows that the upper bound
holds even for functions $f_t$ depending on $t$ provided 
that it is non-negative and satisfies
\begin{equation}
\limsup_{t\to\infty}\frac{1}{a\bigl(\frac{t}{b(t)}\bigr)} 
\sup_{x\in{\rm supp}f_t}
E_x\left[\int_0^{\frac{t}{b(t)}} f_t(S_u){\rm d}u\right]\le 1. 
\end{equation}
This version will be used in the proof of the upper bound. 
See \ref{Chen2} for a proof. 
\end{remark}
We apply this to $f=1_{\{0\}}$ so that $\int_0^t f(S_u){\rm d}u=\ell_t(0)$
and 
\begin{equation}
 a(t)=
\begin{cases}
 c \sqrt{t},&d=1,\\
 t^{o(1)},&d\ge 2.
\end{cases}
\end{equation}
The choice 
\begin{equation}
 b(t)=
\begin{cases}
t^{2\rho-\frac{2\mu}{\alpha}-1+2\epsilon},&d=1,\\
t^{\rho-\frac{d\mu}{\alpha}+2\epsilon},&d\ge 2
\end{cases}
\end{equation}
makes
\begin{equation}
a\left(\frac{t}{b(t)}\right)b(t)=
\begin{cases}
c\sqrt{tb(t)}=ct^{\rho-\frac{\mu}{\alpha}+\epsilon},&d=1,\\
t^{o(1)}b(t)^{1+o(1)}\le t^{\rho-\frac{d\mu}{\alpha}+\epsilon},&d\ge 2
\end{cases}
\end{equation}
for sufficiently large $t$ and hence we get, with $\lambda=\frac12$ (say),
\begin{equation}
P_0\left(\ell_t(0)\ge \frac{c}{2}t^{\rho-\frac{{d}\mu}{\alpha}+\epsilon}\right)
 \ge 
\begin{cases}
\exp\set{-ct^{2\rho-\frac{2\mu}{\alpha}-1+2\epsilon}},&d=1,\\
\exp\set{-ct^{\rho-\frac{d\mu}{\alpha}+2\epsilon}},&d\ge 2,\\
\end{cases}
\end{equation}
By doing the trivial change of variable $t\mapsto {\frac{t}4}$ and 
slightly changing $\epsilon$, we obtain 
\begin{equation}
 P_0\left(A_t\ge t^\rho, {S_t=0}\right)
\ge 
\begin{cases}
\exp\left\{-t^{2\mu-1+\epsilon}-
 t^{2\rho-\frac{2\mu}{\alpha}-1+2\epsilon}\right\}&d=1,\\
\exp\left\{-t^{2\mu-1+\epsilon}-
 t^{\rho-\frac{d\mu}{\alpha}+2\epsilon}\right\}&d\ge 2. 
\end{cases}
\end{equation}
Optimizing over $\mu$, that is attained at 
\begin{equation}
 \mu=
\begin{cases}
\frac{\alpha\rho}{\alpha+1},&d=1,\\[5pt]
\frac{\alpha(\rho+1)}{2\alpha+d},&d\ge 2, 
\end{cases}
\label{mu-opt}
\end{equation}
{which are in $(\frac12,1]$ in this regime,} 
we find the desired lower bound. \\

\noindent\underline{\textsc{The third regime
$\rho<\left(\frac{\alpha+1}{2\alpha}1_{d=1}+\frac{d}{2\alpha}1_{d\ge 2}\right)\vee 1$}}: {We first give the proof for $d=1$.} Let us start with the case
$\alpha>1$ and $\rho<1$. In this case, $P_0(A_{{\frac{t}2}}>t^\rho)$ 
converges to one by the law of large numbers~\eqref{LLN}. 
{Thus for sufficiently large $M>0$, we have
$P_0\left(A_{\frac{t}2}>t^\rho, S_{\frac{t}2}<Mt^{\frac12}\right)\ge \frac12$.
By using the Markov property and the local central limit theorem, we conclude 
that 
\begin{equation}
\begin{split}
&  P_0\left(A_t>t^\rho,  S_t=0\right)\\
&\quad \ge 
 P_0\left(A_{\frac{t}2}>t^\rho, 
 S_{\frac{t}2}<Mt^{\frac12}, S_t=0\right)\\
&\quad \ge t^{-\frac12-\epsilon} 
\end{split}
\end{equation}}
for any $\epsilon>0$. 

Next, for $\alpha\le 1$ and 
$\rho<\frac{\alpha+1}{2\alpha}$, we let the random walk go to the 
highest point of $z$-field in $[-t^{\frac12},t^{\frac12}]$ within time 
${\frac{t}4}$, leave the local time larger than $t^{\frac12}$ there 
{by the time $\frac{t}2$ and come back to the origin at $t$}. 
Since the above highest value of $z$ is $t^{\frac1{2\alpha}+o(1)}$ $\P$-almost
surely, we have $A_t\ge t^{\frac{1}{2}+\frac1{2\alpha}+o(1)}$. 
On the other hand, the probability of this event is bounded from below by a 
power of $t$ by the local central limit theorem for $S_t$ and the central 
limit theorem for $\ell_t(0)$ in Darling--Kac~\cite{DK57}.

{In higher dimensions, if $\alpha>\frac{d}2$ and $\rho<1$ 
one can follow the same
strategy as above. When $\alpha\le\frac{d}2$ and $\rho<\frac{d}{2\alpha}$, 
note that the highest point of $z$-field in $[-t^{\frac12},t^{\frac12}]^d$
is $t^{\frac{d}{2\alpha}+o(1)}\gg t^\rho$ $\P$-almost surely. 
Thus it suffices to let the random walk stay there for the unit time 
instead of $t^{\frac12}$. The rest of the argument is the same. 
}
\subsection{Upper bound of Theorem~\ref{RWRS}: the second regime}
\label{spec}
In the second regime, that is, when $\rho> \frac{\alpha+d}{\alpha}$, the proof
of the upper bound is very simple. We set
\begin{equation}
\mu=
\frac{\alpha(\rho-1)}{d}-\epsilon>1
\end{equation}
for small $\epsilon>0$ so that
\begin{equation}
\max_{|x|\le t^\mu}z(x)=t^{\frac{d\mu}{\alpha}+o(1)}<t^{\rho-1}
\end{equation}
as $t$ tends to infinity, $\P$-almost surely. This implies that 
$\{\max_{0\le u\le t}|S_u|\le t^\mu\}\cap\{A_t\ge t^\rho\}=\emptyset$ 
(up to a $\P$-null set)
for sufficiently large $t$. Therefore, 
\begin{equation}
\begin{split}
 P_0\left(\int_0^tz(S_u){\rm d}u\ge t^\rho\right)
 &\le P_0\left(\max_{0\le u\le t}|S_u|> t^\mu\right)\\
 &\le \exp\set{-t^{\mu+o(1)}}, 
\end{split}
\end{equation}
where we have used Lemma~\ref{RWHK} and the reflection principle. 

\subsection{Upper bound of Theorem~\ref{RWRS}: the first regime}
In the proof of the lower bound, we concentrate on the contribution from 
the highest peak of $z$-field. 
What remains to show is that the contributions from lower values of
$z$ are negligible --- more precisely, it is much harder for the
random walk to leave the right amount of occupation time on lower
level sets of $z$. 
Throughout this subsection, we fix
\begin{equation}
\mu= 
\begin{cases}
\frac{\alpha\rho}{\alpha+1},
&d=1\textrm{ and }\rho\in\left[\frac{\alpha+1}{2\alpha}\vee 1,
\frac{\alpha+1}{\alpha}\right],\\[5pt]
\frac{\alpha(\rho+1)}{2\alpha+d},
&d\ge 2\textrm{ and }\rho\in\left[\frac{d}{2\alpha}\vee 1,
\frac{\alpha+d}{\alpha}\right],
\end{cases}
\label{mu-1}
\end{equation}
which are the optimal values found in the proofs of lower bounds
(cf.~\eqref{mu-opt}). 
In view of~\eqref{extreme}, we may replace $z$ by 
\begin{equation}
 \tilde{z}(x)=\left(z(x)\wedge t^{\frac{d\mu}{\alpha}+\epsilon}\right)
1_{\{|x|\le t^\mu\}}
\end{equation}
for a small fixed $\epsilon>0$.

We define the level sets of $\tilde z$-field by
$\mathcal{H}_0=\Z^d$ and
\begin{equation}
 \mathcal{H}_k={\mathcal{H}_k(t)=}
\set{x\in\Z^d\colon \tilde z(x)\ge t^{k\epsilon}}
\quad \textrm{for }k\ge 1.
\end{equation}
Note that $\mathcal{H}_k=\emptyset$ if 
$k> K:={[\frac{d\mu}{\epsilon\alpha}]}$. 
Our starting point is the following obvious bound: 
\begin{equation}
\begin{split}
  P_0\left(A_t\ge t^\rho\right)
&\le \sum_{k=0}^K
 P_0\left(\ell_t(\mathcal{H}_k\setminus \mathcal{H}_{k+1})
 \ge \frac{1}{K+1}t^{\rho-(k+1)\epsilon}\right)\\
&\quad +P_0\left(\max_{0\le u\le t}|S_u| > t^\mu\right).
 \label{slicing}
\end{split}
\end{equation}
Thanks to {Lemma~\ref{RWHK}}, we have
\begin{equation}
P_0\left(\max_{0\le u\le t}|S_u| > t^\mu\right)=
  \exp\set{-t^{2\mu-1+o(1)}}
\end{equation}
(recall $\mu\le 1$) and hence the second term on the right hand side is 
harmless.

We bound each summand in the first term by applying the version of 
Theorem~\ref{Chen} mentioned in Remark~\ref{Chen-ver} to 
$f_t=1_{\mathcal{H}_k}$ {(see also \ref{Chen2})}. 
The input required in the upper bound is the asymptotics of $a(\frac{t}{b_t})$. 
(Note that when $f$ depends on $t$, we cannot relate this to that of $a(t)$.)
Therefore we start by studying
\begin{equation}
E_x\left[\int_0^{t^\eta}f(S_u){\rm d}u\right]=E_x[\ell_{t^\eta}(\mathcal{H}_k)]
\end{equation}
for $\eta>0$.
\begin{lemma}
\label{lem1}
For any $\eta>0$ and $d\ge 1$, there exists a positive constant $c$ such that 
the $\P$-probability of  
\begin{equation}
 \sup_{|x|\le t^\mu}E_x[\ell_{t^\eta}(\mathcal{H}_k)]
\le c
\begin{cases}
t^{\frac{\eta}2}(\log t)^3,&d=1\textrm{ and } 
  k\epsilon>\frac{\eta}{2\alpha}, \\[5pt]
 t^\epsilon,&d\ge 2\textrm{ and } k\epsilon>\frac{\eta}{\alpha}
\end{cases}
\end{equation}
is bounded {from below by $1-c\exp\{-c^{-1}(\log t)^2\}$.} 
In particular, the above holds for all sufficiently large $t$, 
$\P$-almost surely. 
\end{lemma}
\begin{proof}
Let us first explain the outline of the proof. It is straightforward to check
that the mean $\E E_x[\ell_{t^\eta}(\mathcal{H}_k)]$ satisfies the desired
bound. Then we show a concentration around the mean, which should not be hard
since $E_x[\ell_{t^\eta}(\mathcal{H}_k)]$ is a linear functional of 
i.i.d.~random variables 
{$\{1_{\{\tilde{z}(y)\ge t^{k\epsilon}\}}\}_{y\in\Z^d}$}. 
However, if we apply the martingale difference method directly
to this family, then it turns out that we have too many terms. 
It is necessary (and natural) to take into account that these random 
variables degenerate as $t\to\infty$. To this end, we shall divide 
$[-t^\mu,t^\mu]^d$ into smaller cubes $\{I_j\}_{j\ge 1}$ in such a way that 
$\P(\#I_j\cap \mathcal{H}_k\ge 1)$ is bounded away from 0 and 1. 
Then we use a martingale difference type method (McDiarmid's inequality, 
Theorem~6.5 in~\cite{BLM13}) to $\{I_j\cap \mathcal{H}_k\}_{j\ge 1}$. 
That $\#I_j\cap \mathcal{H}_k$ is unbounded causes a little problem
but a simple truncation argument resolves it. This is the role of
\emph{trimmed sets} defined below.

Now we start the proof. We prove the lemma for the special case
\begin{equation}
 \P(z(x)> r)=r^{-\alpha}\wedge 1
\end{equation}
and indicate how to deal with the general case in the proof. 
For any $x\in \Z^d$, 
\begin{equation}
\begin{split}
\E E_x[\ell_{t^\eta}(\mathcal{H}_k)]
&=\int_0^{t^\eta} \sum_{y\in \Z^d} p_u(x,y)\P(\tilde{z}(y)\ge t^{k\epsilon}){\rm d}u\\
&\le\int_0^{t^\eta} \sum_{y\in \Z^d} p_u(x,y)t^{-\alpha k\epsilon}{\rm d}u\\
&\le c
\begin{cases}
t^{\frac\eta2},&d=1,\\
t^{\epsilon},&d\ge 2,
\end{cases}
\label{mean}
\end{split}
\end{equation}  
where we have used the assumptions on $k\epsilon$. 

Next we show a concentration of $E_x[\ell_{t^\eta}(\mathcal{H}_k)]$ 
around its mean. Let $I_j$ be the cube 
$t^{\frac{\alpha k \epsilon}{d}}(j+[0, \frac12)^d)$ ($j\in\frac12\Z^d$) and 
define a \emph{trimmed} set $\underline{\mathcal{H}}_k\subset \mathcal{H}_k$
by keeping only at most $(\log t)^2$ points in each $I_j$. {(The way how to 
choose these points does not matter in what follows.)}
It is simple to check that
\begin{equation}
 \begin{split}
\P(\underline{\mathcal{H}}_k\neq \mathcal{H}_k)
&\le \sum_{j\in \frac{1}{2}\Z^d}\P(\#(\mathcal{H}_k\cap I_j)> (\log t)^2)\\
&\le 2^{-{(\log t)^2}}\#\set{j\in\frac12\Z^d: I_j
\cap [-t^\mu, t^\mu]^{d}
\neq \emptyset}\\
&\le 2^{-{\frac{1}{2}(\log t)^2}} 
\label{no-trim}
 \end{split}
\end{equation}
by using the bound
\begin{equation}
\begin{split}
 &\P\left(\#(\mathcal{H}_k\cap I_j)\ge (\log t)^2\right)\\
 &\quad{\le} \binom{|I_j|}{(\log t)^2}\P\left(z(x)>t^{k\epsilon}
  \right)^{{(\log t)^2}}\\
 &\quad\le \left(|I_j|t^{-\alpha k \epsilon}\right)^{(\log t)^2}\\
&\quad=2^{-d(\log t)^2}.
\end{split}
\end{equation}
\begin{remark}
In order to deal with a general distribution of the form~\eqref{ass-tail}, 
we modify the side-length $L_t$ of $I_j$ as 
\begin{equation}
 L_t^d\P\left(z(x)>t^{k\epsilon}\right)\sim\frac{1}{2}\textrm{ as }t\to\infty.
\end{equation}
In the rest of the proof, we replace $t^{\frac{\alpha k\epsilon}{d}}$ by
$L_t$ which is still $t^{\frac{\alpha k\epsilon}{d}+o(1)}$.
\end{remark}
We are going to apply McDiarmid's inequality (Theorem~6.5 in~\cite{BLM13})
to $E_x[\ell_{t^\eta}(\underline{\mathcal{H}}_k)]$, regarding it as a 
function of independent random variables 
$\{I_j\cap \underline{\mathcal{H}}_k\}_{{j\in \frac{1}{2}\Z^d}}$. 
Let $x=0$ for simplicity and define $\mathbb{H}_k$ as the set of all 
possible realizations of $\underline{\mathcal{H}}_k$. 
Then the influence caused by changing the configuration in $I_j$ 
is bounded as
\begin{equation}
\begin{split}
\Delta_j&=\sup_{{\underline{\mathcal{H}}_k,\underline{\mathcal{H}}_k'\in
\mathbb{H}_k\textrm{ which differ only on }I_j}}
 |E_0[\ell_{t^\eta}(\underline{\mathcal{H}}_k)]
-E_0[\ell_{t^\eta}(\underline{\mathcal{H}}_k')]|\\
&\le {\sup\set{
\int_0^{t^\eta} \sum_{y\in Y} p_u(0,y){\rm d}u
\colon Y\subset I_j\textrm{ and } \# Y\le (\log t)^2 }}.
\label{delta}
\end{split}
\end{equation}
Let us consider the case $d=1$ or 2 first. {Lemma~\ref{RWHK}}
implies
{\begin{equation}
 p_u(0,y)\le c(u^{-\frac{d}2}\wedge 1)
\begin{cases}
 \exp\left\{-\frac{|y|^2}{ct^\eta}\right\},& |y|\le {t^\eta},\\
c \exp\left\{-\frac{|y|}{c}\right\},& |y|>{t^\eta}
\end{cases}
\label{HK}
\end{equation}}
for $u\le t^\eta$.
For $|j|\le  1$, we simply replace $y$ in~\eqref{HK} by $0$ and obtain
\begin{equation}
\begin{split}
  \textrm{R.H.S.\ of \eqref{delta}}
 &\le (\log t)^2\int_0^{t^\eta} c(u^{-\frac{d}2}\wedge 1){\rm d}u \\
 &\le c
\begin{cases}
t^{\frac\eta2}(\log t)^2,&d=1,\\
(\log t)^3,&d=2.
\end{cases}
\label{on-diag}
\end{split}
\end{equation}
Next, for $1<|j|\le t^{\eta-\frac{\alpha k\epsilon}{d}}$, 
since the worst case is to have all the points {of $Y\subset I_j$} closest 
to the origin, we have
\begin{equation}
\begin{split}
\textrm{R.H.S.\ of \eqref{delta}}
&\le (\log t)^2\int_0^{t^\eta} c(u^{-\frac{d}2}\wedge 1)
\exp\left\{-\frac{t^{\frac{2\alpha k \epsilon}{d}-\eta}(|j|-1)^2}{c}\right\}{\rm d}u\\ 
&\le c
\exp\left\{-\frac{t^{\frac{2\alpha k\epsilon}{d}-\eta}(|j|-1)^2}{c}\right\}
\begin{cases}
t^{\frac\eta2}(\log t)^2,&d=1,\\
(\log t)^3,&d=2
\end{cases}
\label{off-diag}
\end{split}
\end{equation}
{by using the first line of \eqref{HK}.}
Finally for $|j|> t^{\eta-\frac{\alpha k\epsilon}{d}}{\vee 1}$, 
one can see that 
\begin{equation}
\begin{split}
\textrm{R.H.S.\ of \eqref{delta}}
&\le \exp\{-c|j|\}
\label{Poisson}
\end{split}
\end{equation}
{again by considering the worst case and using the second line 
of~\eqref{HK}.} The sum of this part is always bounded. 

Now we can complete the case $d=1$. 
For $k\epsilon >\frac\eta{2\alpha}$, the \emph{off-diagonal} 
part~\eqref{off-diag} is stretched exponentially small and 
\begin{equation}
\begin{split}
\sum_{j}\Delta_j^2&\le \sum_{{|j|\le 1}}
  {[\textrm{R.H.S.\ of \eqref{on-diag}}]}^2
 +\sum_{j\ge 2}  {[\textrm{R.H.S.\ of \eqref{off-diag}}]}^2+O(1)\\
&=ct^\eta(\log t)^4+O(1).
\end{split}
\end{equation}
Therefore, McDiarmid's inequality implies
\begin{equation}
\begin{split}
&\P\left(|E_0[\ell_{{t^\eta}}(\underline{\mathcal{H}}_k)]
 -\E E_0[\ell_{{t^\eta}}(\underline{\mathcal{H}}_k)]|>t^{\frac\eta2}(\log t)^3\right)\\
 &\quad\le c\exp\left\{-\frac{t^\eta(\log t)^6}{c\sum_j\Delta_j^2}\right\}\\
 &\quad=c\exp\set{-\frac{1}{2c}(\log t)^2}.
\end{split}
\end{equation}
Using~\eqref{mean} and the fact that $\#\mathcal{H}_k$ is
at most polynomial in $t$ (recall that $\mathcal{H}_k$ are subsets of 
$[-t^\mu, t^\mu]$ by the definition of $\tilde z$), 
we are done. 

The argument for the case $d= 2$ is very similar to $d=1$. 
For $k\epsilon >\frac\eta{\alpha}$, the part~\eqref{off-diag} is again 
stretched exponentially small and hence 
\begin{equation}
\begin{split}
\sum_{j}\Delta_j^2&\le \sum_{|j|\le 1}
  {\textrm{[R.H.S.\ of \eqref{on-diag}]}}^2
 +\sum_{j\ge 2}  {\textrm{[R.H.S.\ of \eqref{off-diag}]}}^2+O(1)\\
&=c(\log t)^6+O(1).
\end{split}
\end{equation}
This bounds and McDiarmid's inequality yield 
\begin{equation}
\P\left(|E_0[\ell_t(\underline{\mathcal{H}}_k)]
 -\E E_0[\ell_t(\underline{\mathcal{H}}_k)]|>(\log t)^4\right)
 \le c\exp\set{-\frac{1}{2c}(\log t)^2}
 \end{equation}
and the rest is the same as before.

Finally we settle the case $d\ge 3$. This case is simpler since we know
\begin{equation}
\int_0^{t^\eta} p_u(0,y){\rm d}u
\le \int_0^\infty p_u(0,y){\rm d}u
\le c (|y|+1)^{2-d}. 
\end{equation}
Using this in~\eqref{delta}, we obtain
\begin{equation}
\begin{split}
\sum_j \Delta_j^2&\le (\log t)^4 \sum_j  \sup_{y\in I_j}
 \left(\int_0^{t^\eta} p_u(0,y){\rm d}u\right)^2\\
&\le c(\log t)^4 \sum_j 
\left(t^{\frac{\alpha k \epsilon}{d}}|j|+1\right)^{4-2d}. 
\end{split}
\end{equation}
This is always $O((\log t^4))$ and the rest is routine. 

{The second part of the lemma follows from the first one and the
Borel--Cantelli lemma.} 
\end{proof}
We use this lemma, together with Theorem~\ref{Chen}, to bound
\begin{equation}
P_0\left(\ell_t(\mathcal{H}_k\setminus \mathcal{H}_{k+1})
>\frac{1}{{K+1}}t^{\rho-(k+1)\epsilon}\right)
\le  P_0\left(\ell_t(\mathcal{H}_k)
>\frac{1}{{K+1}}t^{\rho-(k+1)\epsilon}\right).
\label{slice-wise}
\end{equation}
Note first that this probability is zero if 
$k<\frac{\rho-1}{\epsilon}-1$ as the total mass of $\ell_t$ is $t$. 
Hence we are only concerned with
\begin{equation}
 \rho-1-\epsilon\le k\epsilon\le \frac{d\mu}{\alpha}+\epsilon.
\label{restriction2}
\end{equation}
Let us start with the case $d=1$. 
We choose $b(t)=t^{2\rho-1-2(k+2)\epsilon}$ and $\eta=2(1+(k+2)\epsilon-\rho)$.
{Thanks to~\eqref{restriction2}}, one can verify that 
$k\epsilon>\frac{\eta}{2\alpha}$ for sufficiently small $\epsilon$. 
This allows us to use Lemma~\ref{lem1} to obtain
\begin{equation}
{4}a\left(\frac{t}{b(t)}\right)b(t)<\frac1{{K+1}} t^{\rho-(k+1)\epsilon}. 
\end{equation}
Therefore, Theorem~\ref{Chen} yields the bound 
\begin{equation}
\begin{split}
 & P_0\left(\ell_t(\mathcal{H}_k\setminus \mathcal{H}_{k+1})
 >\frac{1}{{K+1}}t^{\rho-(k+1)\epsilon}\right)\\
 &\quad\le P_0\left(\ell_t(\mathcal{H}_k\setminus \mathcal{H}_{k+1})
 >{4}a\left(\frac{t}{b(t)}\right)b(t)\right)\\
 &\quad\le \exp\{-b(t)\}\\
 &\quad= \exp\left\{-ct^{2\rho-1-2(k+2)\epsilon}\right\}.
\end{split}
\label{mon-1D}
\end{equation}
Since this is increasing in $k$, the largest upper bound is for $k=K$ 
and then $b(t)=t^{p(\alpha,\rho)-c\epsilon}$. 
Now all the terms on the right hand
side of~\eqref{slicing} are bounded by 
$\exp\{-t^{p(\alpha,\rho)-c\epsilon}\}$ and the proof is completed. 

The case $d\ge 2$ is almost the same. We choose 
$\eta=1-\rho+(k+3)\epsilon$ and $b(t)=t^{\rho-(k+3)\epsilon}$. 
One can check that this satisfies $k\epsilon>\frac{\eta}{\alpha}$ with the
help of~\eqref{mu-1} and hence by Lemma~\ref{lem1},
\begin{equation}
{4}a\left(\frac{t}{b(t)}\right)b(t)<\frac1{{K+1}} t^{\rho-(k+1)\epsilon}. 
\end{equation}
Then, just as above, we obtain
\begin{equation}
P_0\left(\ell_t(\mathcal{H}_k\setminus \mathcal{H}_{k+1})
 >\frac{1}{{K+1}}t^{\rho-(k+1)\epsilon}\right)
= \exp\left\{-ct^{\rho-(k+3)\epsilon}\right\}
\label{mon-HD}
\end{equation}
which is largest and matches~\eqref{rwrs} at $k=K$.
\begin{remark}
\label{rem:AC}
Let us comment on the technical difference between our proof and the earlier 
works~\cite{AC03a,AC03b,Cas04} that contain quenched tail estimates. 
We focus on the upper bound which is more 
involved in the present model. 

The arguments in~\cite{AC03a,AC03b,Cas04} rely on
the spectral theoretic technics which were developed in the study of 
the so-called parabolic Anderson model. 
More precisely, they obtain the upper bound through the exponential moment 
(with a small parameter $\epsilon_t$) of $A_t$, using 
the fact that it is a solution of $\partial_tu=(\Delta+\epsilon_tz)u$. 
By restricting the problem to a certain macro-box as we have done above
and using the eigenfunction
expansion, the problem is reduced to the behavior of the largest
eigenvalue of $\Delta+\epsilon_t z$ in a large box. Then a localization 
procedure demonstrated in~\cite{GK00} (Proposition~1) 
allows them to bound it by the
maximum of eigenvalues among smaller sub-boxes. A crucial ingredient
to control the maximum of local eigenvalue is a large deviation principle for 
a scaled $z$-field in a suitable function space, Lemma~2.1 in~\cite{AC03b} 
or Theorem~1 in~\cite{Cas04}. 

In our heavy tailed setting, it seems difficult to find a substitute for
the large deviation principle for the $z$-field. The above proof reveals 
that the relevant part of the $z$-field looks like a 
delta function. This makes the choice of function space (or topology) 
a delicate problem. 
\end{remark}
\subsection{Proof of Theorem~\ref{LDP}}
We follow the same strategy as in the previous subsection. Fix 
$c'\in (\E[z(0)], c)$ and bound the large deviation probability as
\begin{equation}
\begin{split}
  P_0\left(A_t\ge ct\right)
&\le  P_0\left(\int_0^t z(S_u)\wedge t^{l\epsilon}{\rm d}u
 \ge c't\right)\\
&\quad+\sum_{k=l}^K
 P_0\left(\ell_t(\mathcal{H}_k\setminus \mathcal{H}_{k+1})
 \ge \frac{c-c'}{K+1}t^{1-(k+1)\epsilon}\right)\\
&\quad +P_0\left(\max_{0\le u\le t}|S_u| > t^\mu\right),
 \label{slicing2}
\end{split}
\end{equation}
where we set $\rho=1$ in the definition of $K$ and $\mu$. The third term is
negligible as before. As for the second term, 
we use the same level sets $\mathcal{H}_k$ and $\eta$ as in the previous 
subsection. Then with $\rho=1$, we see that the assumptions of Lemma~\ref{lem1} 
hold for $k>\frac{2}{\alpha-1}$ ($d=1$) and $k>\frac{3}{\alpha-1}$
($d\ge 2$) and hence~\eqref{mon-1D} and~\eqref{mon-HD} respectively. 
Thus if we set $l$ to be the smallest integer larger than 
$\frac{2}{\alpha-1}$ ($d=1$) and $\frac{3}{\alpha-1}$ ($d\ge 2$), then
the summand with $k=K$ dominates all the others. Since it has the desired 
asymptotics, it remains to show that the first term in \eqref{slicing2} is 
negligible.

In order to simplify the notation, we rename $l\epsilon$ to $\epsilon$ and
$c'$ to $c$. 
We are going to show that for any $\epsilon\in(0,\frac{1}{20})$, 
$\P$-almost surely, 
\begin{equation}
P_0\left(\int_0^t z(S_u)\wedge t^\epsilon {\rm{d}} u
>ct\right)
\le \exp\set{-{t^{1-7\epsilon}}}
\label{lowest-level}
\end{equation}
for large $t$. This proves Theorem~\ref{LDP} since the exponents in
\eqref{ldp} are strictly less than one.
Our argument is again based on~\cite{Che01} (the proof of Theorem~2 there,
presented in Section 3) and hence it is natural to
begin with the following type of bound. 
\begin{lemma}
\label{lem2}
Let $\epsilon>0$. Then $\P$-almost surely, for all sufficiently large $t$, 
\begin{equation}
\sup_{|x|\le t}
\left|E_x\left[\int_0^{{T}} z(S_u)\wedge t^\epsilon{\rm{d}} u\right]
-{T}\E[z(0)]\right|
\le {\frac{T}{\log t}}
\end{equation}
holds for any $T\in [t^{5\epsilon},t]$.
\end{lemma}
\begin{proof}
One can check that it suffices to prove the claim for $z1_{[-2t,2t]^d}$ 
instead of $z$. Henceforth, every configuration appearing below is set to be 
zero outside $[-2t,2t]^d$. The proof is similar to that of Lemma~\ref{lem1}.
The average over the random scenery is close to ${T}\E[z(0)]$. Indeed, 
there exists $c>0$ such that
\begin{equation}
\begin{split}
 \left|\E E_x\left[\int_0^{{T}} z(S_u)\wedge t^{\epsilon}{\rm{d}} u\right]
 - {T}\E[z(0)] \right|&=T\E[(z(0)-t^\epsilon)_+]\\
&\le \frac{cT}{t^{\epsilon(\alpha-1)/2}}
\end{split}\label{expectation-truncation}
\end{equation}
for any $t>0$, uniformly in {$T\in [t^\epsilon,t]$} and $x\in[-t,t]^d$. 
Note that this is much smaller than $T/\log t$ for sufficiently 
large $t>0$ since we are considering $\alpha>1$.
We shall bound the variation of 
$E_x[\int_0^{{T}} z(S_u)\wedge t^{\epsilon}{\rm{d}} u]$
as a functional of $z$. 
For two configurations $z_1$ and $z_2$, by using the Cauchy--Schwarz
inequality and the Chapman--Kolmogorov identity, we have 
\begin{equation}
\begin{split}
&\left|E_x\left[\int_0^{{T}} z_1(S_u){\rm{d}} u\right]
 -E_x\left[\int_0^{{T}} z_2(S_u){\rm{d}} u\right]\right|\\
&\quad = \left|\sum_{y\in{[-2t,2t]^d}}\int_0^{{T}} p_u(x,y)z_1(y){\rm{d}} u
-\sum_{y\in{[-2t,2t]^d}}\int_0^{{T}} p_u(x,y)z_2(y){\rm{d}} u\right|\\
&\quad \le 
\int_0^{{T}}\left(\sum_{y\in{[-2t,2t]^d}}p_u(x,y)^2\right)^{\frac12}
\left(\sum_{y\in{[-2t,2t]^d}}|z_1(y)-z_2(y)|^2\right)^{\frac12} {\rm{d}} u\\
&\quad \le |z_1-z_2|_2 \int_0^{{T}} p_{2u}(x,x)^{\frac12}{\rm{d}} u\\
&\quad \le c{T}^{\frac34}|z_1-z_2|_2.
\end{split}
\end{equation}
This allows us to use the Talagrand's concentration
inequality (Theorem~6.6 in~\cite{Tal96} and {the argument below (4.2) in
the same paper to replace the median by the mean}) to obtain the bound
\begin{equation}
\begin{split}
& \P\left(\left|E_x\left[\int_0^{{T}} z(S_u)\wedge t^\epsilon{\rm{d}} u
\right]
-\E E_x\left[\int_0^{{T}} z(S_u)\wedge t^\epsilon{\rm{d}} u\right]\right|
\ge \frac{T}{4\log t}\right)\\
&\quad  \le c_1\exp\set{-c_2\frac{{T^2}}{T^{\frac32}t^{2\epsilon}(\log t)^2}}.
\end{split}
\end{equation}
{One can check that the argument of the exponential is a positive power of
$t$ for $T>t^{5\epsilon}$ and $\epsilon<\frac{1}{20}$ 
From this and~\eqref{expectation-truncation}, we deduce the discretized bound 
\begin{equation}
\sup_{|x|\le t}
\left|E_x\left[\int_0^{{T}} z(S_u)\wedge t^\epsilon{\rm{d}} u\right]
-{T}\E[z(0)]\right|
\le \frac{T}{4\log t}
\end{equation}
uniformly in $T\in[t^{5\epsilon}, t]\cap\N$} for all sufficiently large 
$t\in\N$,
$\P$-almost surely, by the Borel--Cantelli lemma. This can be extended to 
continuum $T$, with a price of extra factor 2 on the right hand side, by a 
simple monotonicity argument. 
Moreover, as we vary $t$ over an interval 
$[\lfloor t \rfloor,\lfloor t \rfloor+1]$, the $\sup_{|x|\le t}$ is monotone 
and the above integral varies at most $O(T t^{\epsilon-1})$. Hence the above 
bound extends to continuum $t$ with another extra factor 2.
\end{proof}
Now we repeat the argument in the proof of Theorem~2 in~\cite{Che01}. 
We write $\lambda={t^{-6\epsilon}}$
and $f(x)=z(x)\wedge t^{\epsilon}$ to simplify the notation. 
Let $P_uf(x)=E_x[f(S_u)]$ and 
\begin{equation}
u_\lambda(x)=\int_0^\infty e^{-\lambda u}P_uf(x){\rm{d}} u
=(\lambda-\A)^{-1}f(x), 
\end{equation}
where $\A={(2d)^{-1}\Delta}$ denotes the generator of the continuous time 
simple random walk. {Note that $\P$-almost surely, $f$ is not identically 
zero and hence $u_\lambda(x)>0$ for all $x\in\Z^d$.}
By definition we have  
\begin{equation}
 -\A u_\lambda(x)=u_\lambda(x)(v_\lambda(x)-\lambda),
\end{equation}
with $v_\lambda={f}{u_\lambda}^{-1}$, that is, $u_\lambda$ is harmonic for 
$\A+(v_\lambda-\lambda)$ and consequently
\begin{equation}
 u_\lambda(x)=e^{-\lambda t}E_x\left[ u_\lambda(S_t)\exp\set{\int_0^t
v_\lambda(S_u){\rm{d}} u}\right].
\label{FK}
\end{equation}
{By using Lemma~\ref{lem2} and the obvious bound 
$\int_0^T P_uf(x){\rm{d}} u\le Tt^\epsilon$ for $T\ge 0$, 
we obtain 
\begin{equation}
 u_\lambda(x)\sim \frac{\E[z(0)]}{\lambda}
\label{Tauber}
\end{equation}
as $\lambda={t^{-6\epsilon}}\downarrow 0$ uniformly in $|x|\le t$.
Indeed, integration by parts shows
\begin{equation}
\begin{split}
  u_\lambda(x)&=\left[e^{-\lambda u}\int_0^u P_sf(x){\rm d}s\right]_{u=0}^\infty
 +\int_0^\infty\lambda e^{-\lambda u}\int_0^u P_sf(x){\rm d}s{\rm d}u\\
 &=\int_{t^{5\epsilon}}^t\lambda e^{-\lambda u}u\E[z(0)]{\rm d}u+R(\lambda),
\end{split}
\label{IBP}
\end{equation}
where the first term is asymptotic to $\lambda^{-1}\E[z(0)]$ and $R(\lambda)$ 
is the remainder term. The latter is bounded as 
\begin{equation}
\begin{split}
 |R(\lambda)|&\le\left(\int_0^{t^{5\epsilon}}+\int_t^\infty\right) 
 \lambda e^{-\lambda u}t^\epsilon u{\rm d}u
 +\frac{1}{\log t}\int_{t^{5\epsilon}}^t \lambda e^{-\lambda u} {u}{\rm d}u\\
&=o(\lambda^{-1}).
\end{split}
\end{equation}
}

Using this in~\eqref{FK}, we find 
\begin{equation}
\begin{split}
&
E_0\left[\exp\set{
\frac{\lambda}{\E[z(0)]+\epsilon}\int_0^t f(S_u){\rm{d}} u
}
\colon \max_{0\le u\le t}|S_u|\le t\right]\\
&\quad\le 
E_0\left[ \exp\set{\int_0^t
 v_\lambda(S_u){\rm{d}} u}\colon \max_{0\le u\le t}|S_u|\le t\right]\\
&\quad\le \frac{1}{\min_{|x|\le t}u_\lambda(x)}
E_0\left[ u_\lambda(S_t)\exp\set{\int_0^t
 v_\lambda(S_u){\rm{d}} u}\right]\\
&\quad\stackrel{\eqref{FK}}{\le}
{e^{\lambda t}}\frac{u_\lambda(0)}{\min_{|x|\le t}u_\lambda(x)}\\
&\quad\stackrel{\eqref{Tauber}}{=}
{e^{\lambda t}}(1+o(1))
\end{split}
\end{equation}
as $t$ tends to infinity. 
Applying the Markov inequality, we arrive at
\begin{equation}
\begin{split}
& P_0\left(\int_0^t f(S_u){\rm{d}} u \ge ct, \max_{0\le u\le t}|S_u|\le t\right)\\
&\quad \le 2\exp\set{\lambda t-\frac{ct}{\E[z(0)]\lambda^{-1}}}\\
&\quad = 2\exp\set{-\left(\frac{c}{\E[z(0)]}-1\right)
{t^{1-6\epsilon}}}. 
\end{split}
\end{equation}
As $P_0(\max_{0\le u\le t}|S_u|> t)$ decays exponentially in $t$, 
this completes the proof of \eqref{lowest-level}. 

\section{Proofs for the random walk in random layered conductance}
\label{proof-RCM}
\subsection{Proofs of Theorems~\ref{exp} and~\ref{alg}}
Recall that we write 
$x\in\Z^{1+d}$ as $(x_1, x_2)$ with $x_1\in\Z$ and $x_2\in \Z^d$. 
Also whenever one finds a point in $\R^k$ below, it should be understood 
as a closest lattice point.
\begin{proof}[Proof of Theorem~\ref{exp}]
We give the proof only for $d=1$ since the higher dimensional case is 
almost the same. Without loss of generality we may assume 
$\mathbf{e}=\mathbf{e}_2$. 
Let $S^1$ and $S^2$ be continuous time simple random walks on $\Z$ 
independent of each other (strictly speaking, they have jump rates $\frac12$). 
Then our process has the representation
\begin{equation}
 (X^1_t, X^2_t)_{{t\ge 0}}=(S^1_{A^2_t}, S^2_t)_{{t\ge 0}},
\end{equation}
where the clock process is defined by $A^2_t=\int_0^t z(S^2_u){\rm d}u$. 
This representation allows us to write
\begin{equation}
\begin{split}
P^\omega_0(X_t=t^\delta\mathbf{e}_1+t^\gamma\mathbf{e}_2)
&=P_0^{\otimes 2}(S^1_{A^2_t}=t^\delta, S^2_t=t^\gamma)\\
&=E_0[p_{A^2_t}(0,t^\delta)1_{\{S^2_t=t^\gamma\}}].
\label{reduction}
\end{split}
\end{equation}
The last formula involves only the second random walk $S^2$ and 
henceforth we drop the superscript. \\

\noindent\textsc{\underline{Upper bound}:}
We first use H\"older's inequality to obtain
\begin{equation}
 \text{R.H.S.\ of } \eqref{reduction} 
 \le E_0\left[p_{A_t}(0,t^\delta)^2\right]^{\frac12}
P_0(S_t=t^\gamma)^{\frac12}. 
\end{equation}
By {Lemma~\ref{RWHK}}, we have 
\begin{equation}
P_0(S_t=t^\gamma)= \exp\{-t^{(2\gamma-1)\wedge\gamma +o(1)}\} 
\label{2nd-factor}
\end{equation}
as $t$ tends to infinity. On the other hand, 
\begin{equation}
\begin{split}
E_0\left[p_{A_t}(0,t^\delta)^2\right] 
&\le 
E_0\left[\exp\left\{-\frac{ct^{2\delta}}{A_t}\right\}
1_{\{A_t\ge t^\delta\}}\right]
+E_0\left[\exp\{-ct^\delta\}1_{\{A_t< t^\delta\}}\right]\\
&\le {\sum_{{0\le k<t^M}}
\exp\left\{-\frac{ct^{2\delta}}{{t^\delta+k+1}}\right\}
 P_0\left(t^\delta+k\le A_t< t^\delta+k+1\right)}\\
 &\qquad
+P_0\left(A_t\ge t^M\right)+
E_0\left[\exp\{-ct^\delta\}1_{\{A_t< t^\delta\}}\right].
\end{split}
\end{equation}
The second term can be made negligible by choosing $M>0$ sufficiently 
large and using Theorem~\ref{RWRS}. The first term is bounded by
\begin{equation}
t^M\sup_{{\rho\in[\delta, M]}} 
\exp\{-ct^{2\delta-\rho}\}P_0\left(A_t\ge t^\rho\right)
\end{equation}
with a slightly smaller constant $c>0$. 
We use Theorem~\ref{RWRS} to proceed as
\begin{equation}
\exp\{-ct^{2\delta-\rho}\}P_0\left(A_t\ge t^\rho\right)
 \le \exp\left\{-ct^{2\delta-\rho}-t^{p(\alpha,\rho)-\epsilon}\right\}
\end{equation}
for any $\epsilon>0$, where $p(\alpha,\rho)$ is set to be zero for 
$\rho\le \frac{\alpha+1}{2\alpha}\vee1$. 
Therefore, we arrive at an upper bound of the form
\begin{equation}
E_0[p_{A_t}(0,t^\delta)^2]
\le \exp\left\{-ct^{q(\alpha,\delta)-\epsilon}\right\}
\end{equation}
with 
\begin{equation}
 q(\alpha,\delta)=\delta\wedge \left({\inf_{\rho\in[\delta,M]}}
 p(\alpha,\rho)\vee(2\delta-\rho)\right). 
\label{q2}
\end{equation}
It is a simple matter of checking that this coincides with~\eqref{q}.
\begin{remark}
\label{opt-HK}
We give a brief guide on the final step above for $d=1$. 
In what follows, the constant $M>0$ is chosen sufficiently large.
Figures~\ref{a<1} and~\ref{a>1} might also help. 
Note that $p(\alpha,\rho)$ is 
increasing in $\rho$ whereas $2\delta-\rho$ is decreasing. 
Hence the infimum over 
$\rho\in (0,\infty)$ of $p(\alpha,\rho)\vee(2\delta-\rho)$ takes place at the 
point where 
$p(\alpha,\rho)=2\delta-\rho$, which has a unique solution except for the case
$\alpha>1$ and $\frac{1}{2}<\delta<\frac{\alpha}{\alpha+1}$. 
One can check that the minimizer lies outside $[\delta,M]$ if and only 
if $\alpha>1$ and $\delta>\frac{\alpha}{\alpha-1}$, in which case the infimum
over $\rho\in[\delta,M]$ is attained at $\delta$. Based on these observations, 
one finds that
\begin{equation}
 \begin{cases}
{0,}
&{\delta\le\frac12\vee\frac{\alpha+1}{4\alpha},}\\[5pt]
\frac{4\alpha\delta-\alpha-1}{3\alpha+1},
&{\frac{\alpha}{\alpha+1}\vee\frac12\le}\delta<\frac{2\alpha+1}{2\alpha},\\[5pt]
\frac{\alpha(2\delta-1)}{\alpha+1},
&{\frac{2\alpha+1}{2\alpha}\le \delta <\frac{\alpha}{(\alpha-1)_+},}\\[5pt]
{\frac{\alpha(\delta-1)}{d},}
&{\delta>\frac{\alpha}{(\alpha-1)_+}.}
 \end{cases}
\end{equation}
In the first three cases, the values are smaller than $\delta$ and the 
fourth one is larger than $\delta$. 
Finally, in the remaining case $\alpha>1$ and 
$\frac12<\delta<\frac{\alpha}{\alpha+1}$, we have 
\begin{equation}
 p(\alpha,\rho)\vee(2\delta-\rho)=
\begin{cases}
 2\delta-\rho,& \rho\le 1, \\
 p(\alpha,\rho),& \rho>1 
\end{cases}
\end{equation}
and this takes the minimum value $2\delta-1$ at $\rho=1\in[\delta,M]$. 
\end{remark}
\begin{figure}[tbp]
 \begin{picture}(300,180)(-80,-13)
 \put(-10,-10){$0$}
 \put(0,0){\vector(1,0){200}}
 \put(0,0){\vector(0,1){150}}
 \put(205,0){$\rho$}
 \put(0,155){$h$}
 \put(0,64){\line(1,-1){64}}
 \put(-42,64){$\delta<\frac{\alpha+1}{4\alpha}$}
 \put(0,115){\line(1,-1){115}}
 \put(-80,115){$\frac{\alpha+1}{4\alpha}<\delta<\frac{2\alpha+1}{2\alpha}$}
 \put(70,150){\line(1,-1){110}}
 \put(82,140){$\delta>\frac{2\alpha+1}{2\alpha}$}
 \put(69,-13){$\frac{\alpha+1}{2\alpha}$}
 \put(45,-13){$1$}
 \multiput(50,0)(0,5){11}{\line(0,1){2}}
 \put(135,-13){$\frac{\alpha+1}{\alpha}$}
 \multiput(140,0)(0,5){11}{\line(0,1){2}}
 \put(-13,45){$1$}
 \multiput(0,50)(5,0){29}{\line(1,0){2}}
 \qbezier(72,0)(104,25)(140,50)
 \qbezier(200,70)(170,60)(140,50)
 \dashline{2}(0,0)(150,150)
 \end{picture}
\caption{The case $\alpha\le 1$: The decreasing lines are $h=2\delta-\rho$. 
The increasing dashed line is $h=\rho$. 
The increasing piecewise linear curve, including the flat piece from $(0,0)$ 
to $(\frac{\alpha+1}{2\alpha},0)$, is $h=p(\alpha,\rho)$. 
In this case, the slopes of the latter never exceeds one. It follows that
the point $2\delta-\rho=p(\alpha,\rho)$ is always below $\delta$,
that is, the point of $\rho=2\delta-\rho$.}\label{a<1} 
\end{figure}
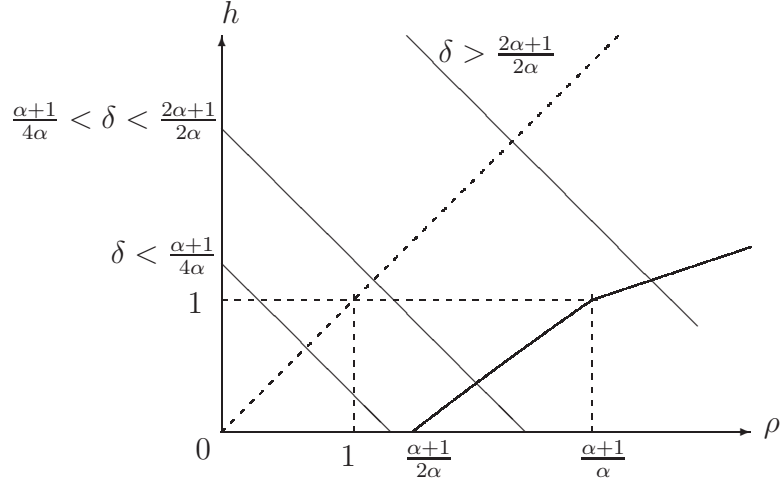
\begin{figure}[tbp]
 \begin{picture}(300,230)(-80,-13)
 \put(0,0){\vector(1,0){200}}
 \put(0,0){\vector(0,1){200}}
 \put(-10,-10){$0$}
 \put(205,0){$\rho$}
 \put(0,205){$h$}
 \put(0,84){\line(1,-1){84}}
 \put(-65,84){$\frac{1}{2}<\delta<\frac{\alpha}{\alpha+1}$}
 \put(0,140){\line(1,-1){140}}
 \put(-80,140){$\frac{\alpha}{\alpha+1}<\delta<\frac{2\alpha+1}{2\alpha}$}
 \put(40,190){\line(1,-1){150}}
 \put(195,35){$\frac{2\alpha+1}{2\alpha}<\delta<\frac{\alpha}{\alpha-1}$}
 \put(140,200){\line(1,-1){60}}
 \put(205,140){$\delta>\frac{\alpha}{\alpha-1}$}
 \put(70,-3){$\bullet$}
 \put(70,-13){$1$}
 \put(105,-13){$\frac{\alpha+1}{\alpha}$}
 \put(70,15){$\circ$}
 \qbezier(73,19)(92,45)(110,70)
 \qbezier(180,200)(145,135)(110,70)
 \multiput(110,0)(0,5){14}{\line(0,1){2}}
 \multiput(0,70)(5,0){22}{\line(1,0){2}}
 \multiput(73,3)(0,5){3}{\line(0,1){2}}
 \multiput(0,18)(5,0){15}{\line(1,0){2}}
 \put(-20,18){$\frac{\alpha-1}{\alpha+1}$}
 \put(-13,65){$1$}
 \dashline{2}(0,0)(200,200)
 \end{picture}
\caption{The case $\alpha> 1$: The decreasing lines are $h=2\delta-\rho$. 
The increasing dashed line is $h=\rho$. 
The increasing piecewise linear curve, including the flat piece from $(0,0)$ 
to $(1,0)$, is $h=p(\alpha,\rho)$. 
When $\alpha>2$, the slope of the last piece exceeds one. Then
the point of $2\delta-\rho=p(\alpha,\rho)$ may be above $\delta$
that is, the point of $\rho=2\delta-\rho$.}
\label{a>1}
\end{figure}
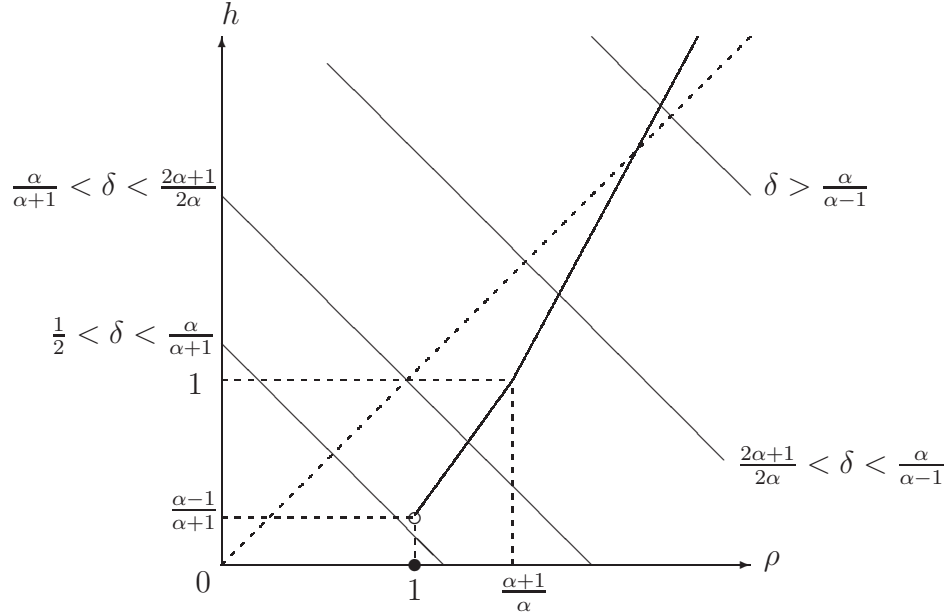
\noindent\textsc{\underline{Lower bound}}:
Since the lower bounds in the first regimes in~\eqref{q} and~\eqref{q-hd}
follow from Theorem~\ref{alg}, we consider the other cases.
In the third and fourth regimes $\frac{\alpha+1}{4\alpha}\vee\frac{\alpha}{\alpha+1}<\delta<\frac{\alpha}{(\alpha-1)_+}$, we set
\begin{equation}
 \rho =
\begin{cases}
\frac{(2\delta+1)(\alpha+1)}{3\alpha+1}{<\frac{\alpha+1}{\alpha}},
&\frac{\alpha+1}{4\alpha}\vee\frac{\alpha}{\alpha+1}<\delta
<\frac{2\alpha+1}{2\alpha},\\[5pt]
\frac{2\delta+\alpha}{\alpha+1}{\ge\frac{\alpha+1}{\alpha}},
&\delta\ge\frac{2\alpha+1}{2\alpha}
\end{cases}
\label{optimal-rho}
\end{equation}
which gives the infimum in the proof of the upper bound. 
{One can verify that this} is larger than 
$\delta\vee 1\vee{\frac{\alpha+1}{2\alpha}}$ {and hence $p(\alpha,\rho)$ 
is to be computed by~\eqref{def-p}}. 
(This can be seen from Figures~\ref{a<1} and~\ref{a>1}.
Indeed, the point where $p(\alpha,\delta)$ crosses $2\delta-\rho$ lies 
in $\{\rho>1{\vee\frac{\alpha+1}{2\alpha}}\}$ and also on the right of 
the dashed line with slope 1.) 
Now we use the first part of Lemma~\ref{RWHK} and the Markov property to obtain
\begin{equation}
\begin{split}
&\text{R.H.S.\ of } \eqref{reduction} \\
 &\quad\ge E_0\left[p_{A_t}(0,t^\delta)\colon S_t=t^\gamma,
\left(\frac{t}2\right)^\rho\le A_t\le t^{M}, S_{\frac{t}2}=0\right]\\
&\quad\ge \exp\set{-t^{2\delta-\rho+o(1)}}
P_0\left({\left(\frac{t}2\right)^\rho\le A_{\frac{t}2}}, 
S_{\frac{t}2}=0\right)P_0(S_{\frac{t}2}=t^\gamma)\\
&\qquad-P_0(A_t> t^{M}). 
\end{split}
\label{LB1}
\end{equation}
Theorem~\ref{RWRS} shows that the last term in~\eqref{LB1} decays faster
than the desired lower bound for sufficiently large $M$.
We use~\eqref{bridge} for the second factor and 
Lemma~\ref{RWHK} for the third factor to bound the first term 
from below by
\begin{equation}
\exp\set{-t^{(2\delta-\rho)\vee p(\alpha,\rho)\vee ((2\gamma-1)\wedge\gamma)
+o(1)}}. 
\end{equation}
This is the desired lower bound.

In second regime $\alpha>1$ and 
$\frac12<\delta\le\frac{\alpha}{\alpha+1}$, note first that we may impose
the constraint $A_t\le t^{\frac{\alpha+1}{\alpha}}$ by Theorem~\ref{RWRS}
as above since 
\begin{equation}
P_0\left(A_t\ge  t^{\frac{\alpha+1}{\alpha}}\right) 
=\exp\left\{-t^{1+o(1)}\right\}
\end{equation} 
decays faster than the desired bound. 
We use~\eqref{bridge} with $\rho=1-\epsilon$
to see that there is $m>0$ such that
\begin{equation}
 P_0\left(t^{1-\epsilon}\le A_\frac{t}{2}, S_{\frac{t}2}=0\right)
\ge t^{-m}
\end{equation}
for large $t$. 
Then, the lower bound follows by the same way as in~\eqref{LB1}.

Finally, in the last case $\delta\ge\frac{\alpha}{(\alpha-1)_+}>1$, 
we impose $t^{1-\epsilon} \le A_t\le t^M$ under which Lemma~\ref{RWHK} yields
\begin{equation}
 p_{A_t}(0,t^\delta)\ge \exp\left\{-t^{\delta+o(1)}\right\}.
\end{equation}
Given this bound, we can argue as before to obtain
\begin{equation}
\begin{split}
\text{R.H.S.\ of } \eqref{reduction} 
&\ge E_0\left[p_{A_t}(0,t^\delta)\colon S_t=t^\gamma,
{t^{1-\epsilon}\le A_t\le 
t^{M}}, 
S_{\frac{t}2}=0\right]\\
&\ge \exp\set{-t^{\delta+o(1)}} t^{-m}
P_0(S_{\frac{t}2}=t^\gamma)-P_0(A_t\ge t^M).
\end{split}
\label{LB2}
\end{equation}
The last term is negligible for sufficiently large $M$ by Theorem~\ref{RWRS}.
Using Lemma~\ref{RWHK} once again, we arrive at the desired lower bound.
\end{proof}
\begin{proof}[Proof of Theorem~\ref{alg}]
This is an easy consequence of~\eqref{bridge} and Theorem~\ref{RWRS}. 
Indeed, we have
\begin{equation}
\begin{split}
&\text{R.H.S.\ of \eqref{reduction}}\\  
&\quad\ge E_0\left[p_{A_t}(0,t^\delta)\colon S_t=t^\gamma,
 \left(\frac{t}{2}\right)^{2\delta}\le A_t \le 
 t^{\frac{\alpha+d}{\alpha}}, S_{\frac{t}2}=0\right]\\
&\quad\ge\inf\set{p_u(0,t^\delta)\colon
u\in \left[\left(\frac{t}2\right)^{2\delta},t^{\frac{\alpha+d}{\alpha}}\right]}
P_0\left(A_{\frac{t}{2}}\ge \left(\frac{t}{2}\right)^{2\delta}, 
S_{\frac{t}{2}}=0\right) P_0(S_{\frac{t}2}=t^\gamma)\\
&\qquad- P_0(A_t\ge t^{\frac{\alpha+d}{\alpha}})
\end{split}
\end{equation}
as before and the last term decays super-polynomially by 
Theorem~\ref{RWRS}. Now if $\gamma$ and $\delta$ satisfy the condition of 
Theorem~\ref{alg}, then the probability 
$P_0( A_{\frac{t}2}\ge (\frac{t}2)^{2\delta}, S_{\frac{t}2}=0)$ is bounded from 
below by a power of $t$ and the transition probabilities 
$p_{u}(0,t^\delta)$ 
for $u\in [t^{2\delta},t^{\frac{\alpha+d}{\alpha}}]$ and 
$P_0(S_{\frac{t}2}=t^\gamma)$ also decay like a negative power of $t$ 
by the local central limit theorem. 
\end{proof}

\subsection{Proof of Theorem~\ref{range}}
We shall present a proof for the third regime in~\eqref{q}
and later indicate how to adapt the argument to the other regimes, 
see Remark~\ref{adapt} below.
Note that in the regimes of concern, the tail 
asymptotics are determined by the probability that the second ($d$-dimensional) 
component of the random walk goes to an extreme point of $z$-field and 
stays there. 
In particular, the vertical displacement plays no role in the exponents. 
The key observation is that the heat kernel bound is violated when 
$z$-field has a atypically large extreme value and this has a polynomial
decaying probability. 

Fix $\rho$ which appears in the lower bound of Theorem~\ref{exp} 
and also the corresponding $\mu$ in the lower bound of Theorem~\ref{RWRS}.
Consider the event 
\begin{equation}
\set{\omega\colon \max_{|x|\le t^{\mu-\frac{r}{2}}}z(x)
\ge t^{\frac{d\mu}{\alpha}+\frac{r}{2}}}.
\label{event-violate}
\end{equation}
This event has a probability 
larger than $t^{-\frac{r(\alpha+d)}{2}+o(1)}$. 
On this event, we let the random walk go to a point that maximizes
$z(\cdot)$ within $|x|\le t^{\mu-\frac{r}{2}}$ and then follow the same 
strategy as in the lower bound of Theorem~\ref{RWRS}.
This strategy gives a slightly better lower bound
\begin{equation}
 P^\omega_0(X_t=t^\delta\mathbf{e}_1)\ge
 \exp\left\{-t^{q(\alpha,\delta)-r+o(1)}\right\}
\end{equation}
and hence we have proved the lower bound
\begin{equation}
 \P(\tau_r>t)\ge t^{-\frac{r(\alpha+d)}{2}+o(1)}.
\label{tau-lower}
\end{equation}

\begin{remark}
\label{adapt}
When $d\ge 2$ and 
$\delta\in(\frac{2\alpha}{2\alpha+d}\vee\frac{d}{4\alpha},\frac{2\alpha+d}{2\alpha}]$ or $d\ge 1$ and 
$\delta\in (\frac{2\alpha+d}{2\alpha},\frac{\alpha}{(\alpha-d)_+})$ 
(the fourth regime), 
the event
\begin{equation}
\max_{|x|\le t^{\mu-\frac{r}{2}}}z(x)
\ge t^{\frac{d\mu}{\alpha}+r}\textrm{ or } 
\max_{|x|\le t^{\mu-r}}z(x)\ge t^{\frac{d\mu}{\alpha}}
\end{equation}
respectively plays the role of~\eqref{event-violate}. 
When $q(\alpha,\delta)\vee ((2\gamma-1)\wedge\gamma)=2\delta-1$ 
(the second regime), the event
\begin{equation}
\max_{|x|\le t^{\delta-\frac{r}2}}z(x)
\ge \begin{cases}
t^{1+\delta+\frac{3r}{2}},& d=1,\\
t^{2(1-\delta+r)},& d\ge 2,
\end{cases} 
\label{exceedance}
\end{equation}
plays the role of~\eqref{event-violate}. 
\end{remark}
\section{Open problems}
We list a few open problems:
\begin{enumerate}
 \item It is of course desirable to get rid of $o(1)$ errors from both 
Theorems~\ref{RWRS} and~\ref{exp}. The proof of Theorem~\ref{RWRS}
tells us that it would be helpful to understand the asymptotics of 
$\sup_{|x|\le R}E_x[\int_0^Tz(S_u){\rm d}u]$ as $R,T\to\infty$ in a 
coupled manner. For this problem, one should be careful 
about the formulation. As observed in~\cite{HMS08}, the result can 
fluctuate in the leading order and then the weak and almost sure 
limit may differ.
 \item It is also interesting to understand how 
$P_0(A_t\ge t^{\frac{\alpha+1}{2\alpha}}b(t))$ ($d=1,\alpha\le 1$) and 
$P_0(A_t\ge t^{\frac{d}{2\alpha}}b(t))$ ($d\ge 2, \alpha\le \frac{d}2$) 
behave for 
a function $b(t)\to\infty$. This possibly leads to the law of iterated
logarithm type result. 
 \item In the case $\E[z(0)]<\infty$, it would be nice to complement 
Theorem~\ref{LDP} by finding an estimate on $P_0(A_t\le ct)$ for
$c<\E[z(0)]$. In contrast to Theorem~\ref{LDP}, the strategy for
the random walk is to stay on a low level set of $z$-field. 
Therefore it is natural to expect that the order of $\log P_0(A_t\le ct)$ 
is different from Theorem~\ref{LDP}. The study of lower deviations is
also related to the \emph{on-diagonal} estimate for the random conductance 
model through 
\begin{equation}
 P^\omega_0(X_t=0)=E_0[p_{A_t}(0,0) \colon S_t=0]
\sim ct^{-\frac{d}{2}}E_0\left[A_t^{-\frac{1}{2}} \,\middle|\, S_t=0\right].
\end{equation}
 \item Concerning the random conductance model, a natural extension is
to make other lines parallel to axis random. It no longer admits the 
time change representation in terms of the random walk in random scenery.
When conductance is 
bounded and uniformly elliptic, it is what is called toy-model 
in~\cite{Bis11}, Section~2.3. The quenched invariance principle
proved there extends to the case where the conductance of every coordinate
axis has finite mean. The heavy tailed case would be much more complicated
and there might be an explosion in finite time.
\end{enumerate}

\appendix

\section{}\label{Chen2}
We include the following version of Chen's result for the sake of completeness. 
The bound and its proof are literary taken from~\cite{Che01}.
As it is a non-asymptotic result, it applies to $f$ which depends on $t$
and one can deduce the version indicated in Remark~\ref{Chen-ver}. 
\begin{proposition}
Suppose $t>0$, $f\colon \Z^d\to[0,\infty)$ be a bounded function and
$a,b\colon (0,\infty)\to (1,\infty)$ satisfy
\begin{equation}
\sup_{x\in{\rm supp}f}E_x\left[\int_0^{\frac{t}{b(t)}} f(S_u){\rm d}u\right]
\le a\left(\frac{t}{b(t)}\right).
\label{ass-Chen2}
\end{equation}
Then for $\lambda>0$, 
\begin{equation}
P_0\left(\int_0^tf(S_u){\rm d}u\ge \lambda a\left(\frac{t}{b(t)}\right)b(t)\right)
\le 2^{\frac12}e^{\frac{1}{24(b(t)-1)}}
\left(\frac{\lambda e}{4}\right)^{-b(t)+1}.
\end{equation}
\end{proposition}
\begin{proof}
Note first that the supremum in~\eqref{ass-Chen2} can be extended to 
$x\in\Z^d$ by using the strong Markov property at the hitting time to
${\rm supp}f$. 
The key lemma is the following non-asymptotic bound for the moments of an 
additive functional. 
\begin{lemma}
\label{MR}
For any $m\in\N$, $t>0$ and a bounded function 
$f\colon \Z^d\to [0,\infty)$, 
\begin{equation}
\sup_{x\in\Z^d}E_x\left[\left(\int_0^tf(S_u){\rm d}u\right)^m\right]
\le m! \sup_{x\in\Z^d}E_x\left[\int_0^t f(S_u){\rm d}u\right]^m. 
\end{equation}
\end{lemma}
This goes back at least to~\cite{Kha59}, see the proof of Lemma~3 there.

In what follows, we write $[b(t)]$ for the integer part of $b(t)$
and 
\begin{equation}
I_k=\int_{\frac{(k-1)t}{[b(t)]}}^{\frac{kt}{[b(t)]}}f(S_u){\rm d}u
\end{equation} 
for $1\le k\le [b(t)]$. Then by the multinomial identity, 
\begin{equation}
\begin{split}
E_x\left[\left(\int_0^t f(S_u){\rm d}u\right)^{[b(t)]}\right]
&= E_x\left[\left(\sum_{k=1}^{[b(t)]}
\int_{\frac{(k-1)t}{[b(t)]}}^{\frac{kt}{[b(t)]}} f(S_u){\rm d}u
\right)^{[b(t)]}\right]\\
&\le \sum_{k_1+\cdots+k_{[b(t)]}=[b(t)]}
\frac{[b(t)]!}{k_1!\cdots k_{[b(t)]}!}
E_x\left[I_1^{k_1}\cdots I_{[b(t)]}^{k_{[b(t)]}}\right].
\label{expansion}
\end{split}
\end{equation}
Applying the Markov property at $\frac{kt}{[b(t)]}$ ($1\le k< [b(t)]$) and
using Lemma~\ref{MR}, we find
\begin{equation}
E_x\left[I_1^{k_1}\cdots I_{[b(t)]}^{k_{[b(t)]}}\right]
\le k_1!\cdots k_{[b(t)]}!\sup_{y\in\Z^d}
E_y\left[\int_0^{\frac{t}{[b(t)]}}f(S_u){\rm d}u\right]^{k_1+\cdots+k_{[b(t)]}}.
\end{equation}
Hence the right hand side of~\eqref{expansion} is bounded by
\begin{equation}
 [b(t)]!\binom{2[b(t)]}{[b(t)]}\sup_{y\in\Z^d}
E_y\left[\int_0^{\frac{t}{[b(t)]}}f(S_u){\rm d}u\right]^{[b(t)]}
\le \frac{(2[b(t)])!}{[b(t)]!}a\left(\frac{t}{[b(t)]}\right)^{[b(t)]}.
\end{equation}
Finally, the Markov inequality together with Stirling's approximation 
(see~\cite{Fel68}, Chapter II-9) yields
\begin{equation}
\begin{split}
&P_x\left(\int_0^tf(S_u){\rm d}u\ge 
\lambda a\left(\frac{t}{[b(t)]}\right)[b(t)]\right)\\
&\quad\le \left(\lambda a\left(\frac{t}{[b(t)]}\right)[b(t)]\right)^{-[b(t)]}
E_x\left[\left(\int_0^t f(S_u){\rm d}u\right)^{[b(t)]}\right]\\
&\quad\le \lambda^{-[b(t)]}[b(t)]^{-[b(t)]}
\frac{\sqrt{2\pi}(2[b(t)])^{2[b(t)]+\frac{1}{2}}
e^{-2[b(t)]+\frac{1}{24[b(t)]}}}
{\sqrt{2\pi}[b(t)]^{[b(t)]+\frac{1}{2}}
e^{-[b(t)]}}\\
&\quad= 2^{\frac12}e^{\frac{1}{24[b(t)]}}
\left(\frac{\lambda e}{4}\right)^{-[b(t)]}.
\end{split}
\end{equation}
\end{proof}
\section*{Acknowledgments}
\noindent
The first author thanks RIMS at Kyoto university for the kind hospitality 
and the financial support.
The second author was partially supported by JSPS KAKENHI Grant 
Number 24740055 and 16K05200. 
{The authors are grateful to a reviewer for pointing out a flaw in 
the proof of the lower bounds in Theorem~\ref{exp} and~\ref{alg}.}

\newcommand{\noop}[1]{}


\begin{thebibliography}{99}

\bibitem{ADS16}
S.~Andres, J.-D. Deuschel, and M.~Slowik.
\newblock Heat kernel estimates for random walks with degenerate weights.
\newblock {\em Electron. J. Probab.}, 21:Paper No. 33, 21 pp., 2016.

\bibitem{AC03b}
A.~Asselah and F.~Castell.
\newblock Large deviations for {B}rownian motion in a random scenery.
\newblock {\em Probab. Theory Related Fields}, 126(4):497--527, 2003.

\bibitem{AC03a}
A.~Asselah and F.~Castell.
\newblock Quenched large deviations for diffusions in a random {G}aussian shear
  flow drift.
\newblock {\em Stochastic Process. Appl.}, 103(1):1--29, 2003.

\bibitem{AC07b}
A.~Asselah and F.~Castell.
\newblock A note on random walk in random scenery.
\newblock {\em Ann. Inst. H. Poincar\'e Probab. Statist.}, 43(2):163--173,
  2007.

\bibitem{AC07a}
A.~Asselah and F.~Castell.
\newblock Random walk in random scenery and self-intersection local times in
  dimensions {$d\ge5$}.
\newblock {\em Probab. Theory Related Fields}, 138(1-2):1--32, 2007.

\bibitem{Bis11}
M.~Biskup.
\newblock Recent progress on the random conductance model.
\newblock {\em Probab. Surveys}, 8:294--373, 2011.

\bibitem{Bor79a}
A.~N. Borodin.
\newblock A limit theorem for sums of independent random variables defined on a
  recurrent random walk.
\newblock {\em Dokl. Akad. Nauk SSSR}, 246(4):786--787, 1979.

\bibitem{Bor79b}
A.~N. Borodin.
\newblock Limit theorems for sums of independent random variables defined on a
  transient random walk.
\newblock {\em Zap. Nauchn. Sem. Leningrad. Otdel. Mat. Inst. Steklov. (LOMI)},
  85:17--29, 237, 244, 1979.
\newblock Investigations in the theory of probability distributions, IV.

\bibitem{BLM13}
S.~Boucheron, G.~Lugosi, and P.~Massart.
\newblock {\em Concentration inequalities}.
\newblock Oxford University Press, Oxford, 2013.
\newblock A nonasymptotic theory of independence, With a foreword by Michel
  Ledoux.

\bibitem{Cas04}
F.~Castell.
\newblock Moderate deviations for diffusions in a random {G}aussian shear flow
  drift.
\newblock {\em Ann. Inst. H. Poincar\'e Probab. Statist.}, 40(3):337--366,
  2004.

\bibitem{CP01}
F.~Castell and F.~Pradeilles.
\newblock Annealed large deviations for diffusions in a random {G}aussian shear
  flow drift.
\newblock {\em Stochastic Process. Appl.}, 94(2):171--197, 2001.

\bibitem{Che01}
X.~Chen.
\newblock Moderate deviations for {M}arkovian occupation times.
\newblock {\em Stochastic Process. Appl.}, 94(1):51--70, 2001.

\bibitem{Che10}
X.~Chen.
\newblock {\em Random walk intersections}, volume 157 of {\em Mathematical
  Surveys and Monographs}.
\newblock American Mathematical Society, Providence, RI, 2010.
\newblock Large deviations and related topics.

\bibitem{CKS99}
E.~Cs{\'a}ki, W.~K{\"o}nig, and Z.~Shi.
\newblock An embedding for the {K}esten-{S}pitzer random walk in random
  scenery.
\newblock {\em Stochastic Process. Appl.}, 82(2):283--292, 1999.

\bibitem{DK57}
D.~A. Darling and M.~Kac.
\newblock On occupation times for {M}arkoff processes.
\newblock {\em Trans. Amer. Math. Soc.}, 84:444--458, 1957.

\bibitem{DD05}
T.~Delmotte and J.-D. Deuschel.
\newblock On estimating the derivatives of symmetric diffusions in stationary
  random environment, with applications to {$\nabla\phi$} interface model.
\newblock {\em Probab. Theory Related Fields}, 133(3):358--390, 2005.

\bibitem{dHS06}
F.~den Hollander and J.~E. Steif.
\newblock Random walk in random scenery: a survey of some recent results.
\newblock In {\em Dynamics \& stochastics}, volume~48 of {\em IMS Lecture Notes
  Monogr. Ser.}, pages 53--65. Inst. Math. Statist., Beachwood, OH, 2006.

\bibitem{Fel68}
W.~Feller.
\newblock {\em An introduction to probability theory and its applications.
  {V}ol. {I}}.
\newblock Third edition. John Wiley \& Sons, Inc., New York-London-Sydney,
  1968.

\bibitem{FMW08}
K.~Fleischmann, P.~M{\"o}rters, and V.~Wachtel.
\newblock Moderate deviations for a random walk in random scenery.
\newblock {\em Stochastic Process. Appl.}, 118(10):1768--1802, 2008.

\bibitem{GKS07}
N.~Gantert, W.~K{\"o}nig, and Z.~Shi.
\newblock Annealed deviations of random walk in random scenery.
\newblock {\em Ann. Inst. H. Poincar\'e Probab. Statist.}, 43(1):47--76, 2007.

\bibitem{GHK06}
N.~Gantert, R.~van~der Hofstad, and W.~K{\"o}nig.
\newblock Deviations of a random walk in a random scenery with stretched
  exponential tails.
\newblock {\em Stochastic Process. Appl.}, 116(3):480--492, 2006.

\bibitem{GK00}
J.~G{\"a}rtner and W.~K{\"o}nig.
\newblock Moment asymptotics for the continuous parabolic {A}nderson model.
\newblock {\em Ann. Appl. Probab.}, 10(1):192--217, 2000.

\bibitem{GP02}
N.~Guillotin-Plantard.
\newblock Large deviations for a {M}arkov chain in a random landscape.
\newblock {\em Adv. in Appl. Probab.}, 34(2):375--393, 2002.

\bibitem{GPPdS14}
N.~Guillotin-Plantard, J.~Poisat, and R.~S. dos Santos.
\newblock A quenched functional central limit theorem for planar random walks
  in random sceneries.
\newblock {\em Electron. Commun. Probab.}, 19:no. 3, 9, 2014.

\bibitem{Kak51}
S.~Kakutani.
\newblock Random ergodic theorems and {M}arkoff processes with a stable
  distribution.
\newblock In {\em Proceedings of the {S}econd {B}erkeley {S}ymposium on
  {M}athematical {S}tatistics and {P}robability, 1950}, pages 247--261.
  University of California Press, Berkeley and Los Angeles, 1951.

\bibitem{KS79}
H.~Kesten and F.~Spitzer.
\newblock A limit theorem related to a new class of self-similar processes.
\newblock {\em Z. Wahrsch. Verw. Gebiete}, 50(1):5--25, 1979.

\bibitem{Kha59}
R.~Khas'minskii.
\newblock On positive solutions of the equation $\mathfrak{A}u + vu = 0$.
\newblock {\em Theory Probab. Appl.}, 4(3):309--318.

\bibitem{KL98}
D.~Khoshnevisan and T.~M. Lewis.
\newblock A law of the iterated logarithm for stable processes in random
  scenery.
\newblock {\em Stochastic Process. Appl.}, 74(1):89--121, 1998.

\bibitem{Kon10}
W.~K\"onig.
\newblock Upper tails of self-intersection local times of random walks: survey
  of proof techniques.
\newblock (1):15--24, 2010.

\bibitem{Li12}
Y.~Li.
\newblock Moderate deviations for stable random walks in random scenery.
\newblock {\em J. Appl. Probab.}, 49(1):280--294, 2012.

\bibitem{Rem00}
B.~Remillard.
\newblock Large deviations estimates for occupation time integrals of
  {B}rownian motion.
\newblock In {\em Stochastic models ({O}ttawa, {ON}, 1998)}, volume~26 of {\em
  CMS Conf. Proc.}, pages 375--398. Amer. Math. Soc., Providence, RI, 2000.

\bibitem{Tal96}
M.~Talagrand.
\newblock A new look at independence.
\newblock {\em Ann. Probab.}, 24(1):1--34, 1996.

\bibitem{HMS08}
R.~van~der Hofstad, P.~M{\"o}rters, and N.~Sidorova.
\newblock Weak and almost sure limits for the parabolic {A}nderson model with
  heavy tailed potentials.
\newblock {\em Ann. Appl. Probab.}, 18(6):2450--2494, 2008.

\bibitem{Zha01}
L.~Zhang.
\newblock Strong approximation for the general {K}esten-{S}pitzer random walk
  in independent random scenery.
\newblock {\em Sci. China Ser. A}, 44(5):619--630, 2001.

\end{thebibliography}
\end{document}